\documentclass[reqno,a4paper,12pt]{amsart}

\usepackage[all,poly]{xy}
\usepackage{amsfonts}
\usepackage[mathcal]{eucal}
\usepackage{eufrak}
\usepackage{amssymb}
\usepackage{amsmath}
\usepackage{mathrsfs}
\usepackage{color}
\usepackage[colorlinks]{hyperref}
\usepackage{enumerate}

\definecolor{citecol}{RGB}{145, 1, 1}

\hypersetup{colorlinks=true,citecolor=citecol,linkcolor=blue,linktocpage=true}

\theoremstyle{plain}
\newtheorem {lemma}{Lemma}[section] 
\newtheorem {theorem}[lemma]{Theorem}

\newtheorem {thm}[lemma]{Theorem}
\newtheorem {corollary}[lemma]{Corollary}

\newtheorem {cor}[lemma]{Corollary}

\newtheorem {prop}[lemma]{Proposition}

\theoremstyle{definition}

\newtheorem {example}[lemma]{Example}
\newtheorem {remdef}[lemma]{Remark-Definition}
\theoremstyle{definition}

\newtheorem{deff}[lemma]{Definition}{}
\newtheorem{conj}[lemma]{Conjecture}

\newcommand{\gr}{\operatorname{gr}}

\newcommand{\im}{\operatorname{im}}

\begin{document}

\title[Williams' conjecture]{WILLIAMS' CONJECTURE HOLDS FOR GRAPHS OF GELFAND-KIRILLOV DIMENSION THREE}

\author{Tran Quang Do}
\address{Institute of Mathematics, VAST, 18 Hoang Quoc Viet, Cau Giay, Hanoi, Vietnam}
\email{tqdo@math.ac.vn}

\author{Roozbeh Hazrat}\address{
Centre for Research in Mathematics and Data Science\\
Western Sydney University\\
Australia}
\email{r.hazrat@westernsydney.edu.au}

\author{Tran Giang  Nam}
\address{Institute of Mathematics, VAST, 18 Hoang Quoc Viet, Cau Giay, Hanoi, Vietnam}
\email{tgnam@math.ac.vn}



\begin{abstract}
A graph of Gelfand-Kirillov dimension three is a connected finite essential graph such that its Leavitt path algebra has Gelfand-Kirillov dimension three. We provide number-theoretic criteria for graphs of Gelfand-Kirillov dimension three to be strong shift equivalent. We then prove that two graphs of Gelfand-Kirillov dimension three are shift equivalent if and only if they are strongly shift equivalent, if and only if their corresponding Leavitt path algebras are graded Morita equivalent, if and only if their graded $K$-theories, $K^{\gr}_0$, are order-preserving $\mathbb{Z}[x, x^{-1}]$-module isomorphic. As a consequence, we obtain that the Leavitt path algebras of  graphs of Gelfand-Kirillov dimension three are graded Morita equivalent if and only if their graph $C^*$-algebras are equivariant Morita equivalent, and  two graphs $E$ and $F$ of Gelfand-Kirillov dimension three are shift equivalent if and only if the singularity categories $\text{D}_{\text{sg}}(KE/J_E^2)$ and $\text{D}_{\text{sg}}(KF/J_F^2)$ are triangulated equivalent.
\medskip


\medskip

\end{abstract}


\maketitle

\section{Introduction} \label{introkai}

One of the fundamental objects of study in symbolic dynamics is called a {\it shift of finite type}, which consists of
sequences indexed by $\mathbb{Z}$ of symbols chosen from a finite set, that do not include certain ``forbidden" finite sequences,
equipped with a shift map which creates the dynamical behaviour. The applications abound, from topological quantum
field theory, ergodic theory, and statistical mechanics to coding and information theory \cite{lindmarcus}. An isomorphism
between two shifts of finite type is called a {\it conjugacy}. Up to conjugacy, every shift of finite type arises from an {\it essential graph} - that is, a finite connected directed graph $E$ with neither sources nor sinks \cite{lindmarcus}. The shift space $X_E$ associated to the graph $E$ is given by the set of bi-infinite paths in $E$, and the natural shift of the paths to the left. This is called an {\it edge shift}.

Determining whether two shifts of finite type $X_E$ and $X_F$ are conjugate is in general a difficult problem, because it requires knowledge of all possible bi-infinite paths in $E$ and $F$. In his seminal paper \cite{williams}, Williams introduced the notions of {\it strong shift equivalence} (SSE) and {\it shift equivalence} (SE), which are more tractable. Williams \cite{williams} showed that two subshifts of finite type $X_E$ and $X_F$ are conjugate  if and only if the adjacency matrices of $E$ and $F$  are strong shift equivalent, if and only if $E$ can be obtained from $F$ by a sequence of in-splittings, out-splittings,  in-amalgamations, and out-amalgamations (see Theorem \ref{willimove} below).

Shift equivalence is a weaker equivalence relation than SSE, and is also more computable. Williams originally asserted in \cite{williams} that SE is equivalent to SSE. Although he identified a flaw in the proof of the direction ``SE implies SSE" a year later \cite{willwrong}, and the above assertion is known as the Williams' conjecture (see Conjecture \ref{Wilconj} below). It took 25 years before the Williams' conjecture was disproved, through counterexamples found by Kim and Roush in \cite{kimroush99}. However, identifying  classes of edge shifts for which SE and SSE are equivalent is an open problem. Even in the case of graphs with two vertices, examples are abound \cite[Example 7.3.13]{lindmarcus} (see Example \ref{SE-SSE-Exam} below). Considering how difficult this is, anything new we can say about
strong shift equivalence while merely assuming shift equivalence is bound to be very interesting. 

Very recently, Cordeiro, Gillaspy, Gon\c{c}alves and the second author \cite{CGGH} proved that Williams' conjecture holds for the class of meteor graphs -- That is, an essential graph consisting of two disjoint cycles and the paths connecting these cycles.  The second author and Pacheco \cite{HP2024} showed that Williams' conjecture holds for the class of essential graphs with three vertices, no parallel edges with no trivial hereditary and saturated subsets.

The aim of this article is to investigate Williams' conjecture for the class of graphs of finite Gelfand-Kirillov dimension. It is well-known that the Gelfand-Kirillov dimension is a power tool to investigate infinite dimensional algebras. In \cite{aajz:lpaofgkd}  Alahmadi, Alsulami, Jain and Zelmanov provided that the Gelfand-Kirillov dimension of the Leavitt path algebra $L(E)$ of a finite graph $E$ is finite if and only if $E$ is a graph with disjoint cycles. In this case,  the Gelfand-Kirillov dimension of the Leavitt path algebra $L_K(E)$, with coefficient in a field $K$,  is a natural number. In particular,  the Gelfand-Kirillov dimension of the Leavitt path algebra of an essential graph is odd number. Motivated by this nice result, we say that a {\it graph of finite Gelfand-Kirillov dimension} is an essential graph for which its Leavitt path algebra has finite Gelfand-Kirillov dimension. In this article, we prove that Williams' conjecture holds for the class of graphs of Gelfand-Kirillov dimension three--It contains the class of meteor graphs as a special case. We refer to Remark-Definition \ref{rem-gwGK3} for an explicit description of graphs of Gelfand-Kirillov dimension three. 

Several tools we use are well known, namely, Williams' graph moves and Krieger's dimension theory (see Section 2.4 below). Our main idea is to make a connection from symbolic dynamics, via the theory of Leavitt path algebras, to the notion of {\it talented monoids} introduced by the second author and H. Li \cite{hazli} as follows: If essential graphs $E$ and $F$ are shift equivalent, i.e., their adjacency matrices $A_E$ and $A_F$ are shift equivalent, then their Krieger's dimension groups are isomorphic, $\Delta_{A^t_E}\cong \Delta_{A^t_F}$ by \cite[Theorem 4.2]{krieger} (see Theorem \ref{kriegeror} below). It is well-known that $\Delta_{A^t_E}\cong K^{\gr}_0(L_K(E))$, the latter being the graded Grothendieck group of the Leavitt path algebra $L(E)$ \cite[Lemma 11]{hazd} (see, also
\cite{arapar}). Therefore, we immediately get a $\mathbb{Z}[x, x^{-1}]$-order isomorphism $K^{\gr}_0(L_K(E))\cong K^{\gr}_0(L_K(F))$.
The positive cone of the graded Grothendieck group $K^{\gr}_0(L_K(E))$ may be described explicitly based on $E$, via the so-called talented monoid $T_E$ of $E$. Thereby, from a shift equivalence of graphs we obtain a $\mathbb{Z}$-monoid isomorphism $T_E\cong T_F$. This
isomorphism provides us control over special elements of the talented monoids and consequently on the geometry of the graphs. 
Thus, a careful analysis of the monoid isomorphism and Williams' graph moves allows us to give number-theoretic criteria for graphs of Gelfand-Kirillov dimension three to be strong shift equivalent and use then these criteria to prove that Williams' conjecture holds for these graphs.

It has already noticed by Cuntz and Krieger in 1980's \cite{CunKri80} that the notion of shifts of finite type and their equivalences can be related to invariances of certain graph $C^*$-algebras. This line of research was pursued and substantially developed by Matsumoto and others (see \cite{Mat, ef} and the references there). The Leavitt path algebras which are the discrete version of graph $C^*$-algebras, were introduced in 2005 and the classification of these algebras and their relations to shift of finite type became an active line of research \cite{lpabook, CortHaz}. Being purely algebraic objects, they facilitate connections with other areas of algebra, such as representation theory and even chip-firing games \cite{AH23, HNam}. One of the main conjectures that would relate these topics is the Graded Morita Classification Conjecture: For any finite graphs without sinks $E$ and $F$, their adjacency matrices are shift equivalent if and only if if their Leavitt path algebras are graded Morita equivalent if and only if their graph $C^*$-algebras are equivariant Morita equivalent. Reformulating this in terms of $K$-theory enables us to eliminate the no-sink assumption from the Conjecture (Conjecture \ref{conjehfyhtr} below): for two finite graphs $E$ and $F$, the Leavitt path algebra $L_K(E)$ is graded Morita equivalent to $L_K(F)$ if and only if there is an order-preserving $\mathbb{Z}[x, x^{-1}]$-module isomorphism $K^{\gr}_0(L_K(E))\cong K^{\gr}_0(L_K(F))$. A similar Conjecture can
be written in the analytic setting \cite{CortHaz}: The graph $C^*$-algebras $C^*(E)$ and $C^*(F)$ are equivariant Morita equivalent if and only if $K^{\mathbb{T}}_0(C^*(E))\cong K^{\mathbb{T}}_0(C^*(F))$ as order-preserving $\mathbb{Z}[x, x^{-1}]$-modules. In this paper we show that these conjectures hold for the class of graphs of Gelfand-Kirillov dimension three.

For a field $K$ and a finite graph $E$, the {\it path algebra} $KE$ of $E$ is the $K$-vector space with a basis consisting all (finite) paths in $E$ and the multiplication of paths is just juxtaposition.
There are close relations between Leavitt path algebras and the representation theory of $KE/J^2_E$, where $J_E$ is the two-sided ideal of $KE$ generated by all edges in $E$. For a finite dimensional $K$-algebra $A$, the {\it singularity category} $\text{D}_{\text{sg}}(A)$ of $A$, due to Buchweitz \cite{Buchweitz} and Orlov \cite{Orlov}, defined as the Verdier quotient category of the bounded derived category of finitely generated left $A$-modules modulo the full subcategory consisting of perfect complexes. It was proved by Chen and Yang~\cite{chen}, using the results of Paul Smith~\cite{Smith1} and Hazrat \cite{hazd}, that for essential graphs $E$ and $F$, $\text{D}_{\text{sg}}(KE/J_E^2)$ is triangulated equivalent to $\text{D}_{\text{sg}}(KF/J_F^2)$ if and only if $L_K(E)$ is graded Morita equivalent to $L_K(F)$. Combining this with \cite[Proposition 15 (3)]{hazd}, we obtain that the triangulated equivalence of $\text{D}_{\text{sg}}(KE/J_E^2)$ and $\text{D}_{\text{sg}}(KF/J_F^2)$ implies the shift equivalence of $A_E$ and $A_F$. Conjecture 8.8.2 of \cite{CortHaz} says that the converse of this assertion  also holds. In this paper we prove that the conjecture holds for the class of graphs of Gelfand-Kirillov dimension three.

The paper is organized as follows: In Section \ref{sec2}, we recall the fundamental concepts of $\Gamma$-monoids, the monoids $M_E$, $T_E$ associated with a directed graph $E$, in-splitting, out-splittings, and Krieger's dimension group to enable our careful analysis of the talented monoid $T_E$, which is a $\mathbb{Z}$-monoid. In Section \ref{sec3}, using talented monoids, we show that the class of graphs with disjoint cycles is closed under shift equivalence and shift equivalence preserves the cycle structure of these graphs (Theorem \ref{corcycle1}). Then, by a careful analysis of graphs moves (I) and (O), we provide number-theoretic criteria for graphs of Gelfand-Kirillov dimension three to be strong shift equivalent (Theorems \ref{numtheo-crite-sse1} and \ref{numtheo-crite-sse2}). In Section \ref{sec4}, based on the previous sections, we show that both Williams' Conjecture and The Graded Classification Conjecture hold for the class of graphs of Gelfand-Kirillov dimension three (Theorem \ref{mainthm}). Consequently, we obtain that the Leavitt path algebras of  graphs of Gelfand-Kirillov dimension three are graded Morita equivalent if and only if their graph $C^*$-algebras are equivariant Morita equivalent (Corollary \ref{maintheo-cor1}), as well as show that for two graphs $E$ and $F$ of Gelfand-Kirillov dimension three, their adjacency matrices are shift equivalent if and only if the singularity categories $\text{D}_{\text{sg}}(KE/J_E^2)$ and $\text{D}_{\text{sg}}(KF/J_F^2)$ are triangulated equivalent (Corollary \ref{maintheo-cor2}).

Throughout we write $\mathbb{N}$ for the set of non-negative integers, and under the usual sum, it is the free
monoid generated by a single element.

\section{Preliminaries: Monoids, Graphs and Algebras}\label{sec2}

\subsection{Monoids, \texorpdfstring{\(\mathbb{Z}\)}{ℤ}-monoids and order ideals}

A \emph{semigroup} is a set with an associative binary operation. Subsemigroups of a semigroup are defined in the usual sense: these are the nonempty subsets closed under the operation.

A \emph{monoid} is a semigroup whose operation has an identity element. A {\em submonoid} of a monoid is a subsemigroup that also contains the identity. Throughout this paper, we are most interested in \emph{commutative monoids}, which are those with a commutative operation. In this case, the operation is written additively (with the symbol $+$) and the unit is denoted as zero ($0$). A commutative monoid $M$ is called \emph{conical} if $x+y=0$ in $M$ implies $x=y=0$. The monoid $M$ is called \emph{cancellative} if $x+a=y+a$ in $M$ implies $x=y$.

Given a commutative monoid $M$, we define the \emph{algebraic preorder} on $M$ by setting $x\leq y$ if $y=x+z$ for some $z\in M$. If $M$ is conical and cancellative, then $\leq$ is a partial order. Conversely, if $\leq$ is a partial order then $M$ is conical.

A \emph{$\Gamma$-monoid}, where $\Gamma$ is an abelian group,  consists of a monoid $M$ equipped with a group action of 
$\Gamma$ by monoid homomorphisms. The image of an element $m\in M$ under the action of a group element $\gamma \in \Gamma$ is denoted by $^\gamma m$. A monoid homomorphism $\phi: M\longrightarrow M'$ is called a {\it $\Gamma$-monoid homomorphism} if $\phi(^\gamma m) = $$^\gamma\phi(m)$ for all $m\in M$ and $\gamma\in \Gamma$.


The set of natural numbers is denoted by $\mathbb{N}=\left\{0,1,2,\ldots\right\}$. Under the usual sum, it is the free monoid generated by a single element. One of the $\mathbb Z$-monoids we encounter in this paper is the following. 

\begin{deff}\label{cycmond}	Let $k$ be a positive integer. The monoid $T=\bigoplus_{i=1}^k \mathbb N$, with the action of $\mathbb Z$ defined by ${}^1(a_1,\dots,a_{k-1},a_k)=(a_k,a_1\dots, a_{k-1})$, is called the \emph{$\mathbb Z$-cyclic monoid of rank $k$}.
\end{deff}

Let $M$ be a  $\Gamma$-monoid and $\gamma\in \Gamma$. The $\Gamma$-monoid $M$ is called {\it $\gamma$-periodic} if $^\gamma m = m$ for all $m\in M$. A {\it $\Gamma$-ordered ideal} of $M$ is a submonoid $I$ of $M$ which is closed under the action of $\Gamma$ and it is hereditary in the sense that $x\le y$ and
$y\in I$ imply $x\in y$. A nonzero $\Gamma$-monoid $M$ is called {\it simple} if the only order-ideals of $M$ are $0$ and $M$.

Let $I$ be an order-ideal of $M$. Definite an equivalence relation $\sim_I$ on $M$ as follows: For $a, b\in M$, there exist $x, y\in I$ such that $a + x = b +y$. This is a congruence relation and thus one can form the quotient $\Gamma$-monoid $M/\sim$ which we will denote by $M/I$.

For $\{a_1, \dots a_k\} \subseteq  M$, we denote the  $\Gamma$-order ideal of $M$ generated by the elements $a_i$ by $\langle a_1, \dots, a_k \rangle $. It is easy to see that 
\[\langle a_1, \dots, a_k \rangle=\Big \{ x \in M \mid x \leq \sum_{(i_1, \dots, i_k)\in \Gamma^k} {}^{\gamma_1} a_{1} + \dots + {}^{\gamma_k} a_{k} \Big \}.\]

\subsection{Graphs and associated monoids}

A (directed) graph $E = (E^0, E^1, r, s)$ consists of two disjoint sets $E^0$ and $E^1$, called \emph{vertices} and \emph{edges} respectively, together with two maps $r, s: E^1 \longrightarrow E^0$.  The vertices $s(e)$ and $r(e)$ are referred to as the \emph{source} and the \emph{range} of the edge~$e$, respectively. A graph $E$ is called {\it row-finite} if $|s^{-1}(v)|< \infty$ for all $v\in E^0$. A graph $E$ is called {\it finite} if both sets $E^0$ and $E^1$ are finite.

A {\em sink} in a graph $E$ is a vertex $v \in E^0$ with $|s^{-1}(v)| = 0;$ a {\em source} is a vertex $v \in E^0$ with $|r^{-1} (v)| =0.$  A  finite graph is {\em essential} if it is a graph with neither sources nor sinks.

Let $E$ be a  graph. A (finite) \emph{path} in $E$ is a string
$p=e_1\cdots e_n$ of edges $e_i\in E^1$ such that $r(e_i) = s(e_{i+1})$ for all $i$. The \emph{length} of the path $p = e_1 \cdots e_n$ is $n$, and is denoted by $|p|$. 
The source and range maps on edges are extended to paths as
\[s(e_1\cdots e_n)=s(e_1)\qquad\text{and}\qquad r(e_1\cdots e_n)=r(e_n).\]
Vertices are  regarded as paths of length $0$, with each vertex coinciding with its source and its range. We denote by $\text{Path}(E)$ the set of all paths in $E$.

A vertex $v\in E^0$ is said to \emph{lie} on a path $p$ if $v$ is the source or the range of one of the edges which comprise $p$. The set of vertices that lie on $p$ is denoted by $p^0$.

A path  $p= e_{1} \cdots e_{n}$ of positive length is a \textit{cycle based at} the vertex $v$ if $s(p) = r(p) =v$ and the vertices $s(e_1), s(e_2), \hdots, s(e_n)$ are distinct. An edge $f\in E^1$ is called an \emph{exit} of a cycle $e_1\cdots e_n$ if there is an $1\leq i \leq n$ such that $s(f)=s(e_i)$ and $f\neq e_i$.

A subset $H$ of $E^0$ is called \textit{hereditary} if $r(e)\in H$ implies $s(e)\in H$ for all $e\in E^1$. And $H$ is called \textit{saturated} if whenever $v$ is a regular vertex in $E^0$ with the property that $s(r^{-1}(v)) \subseteq H$, then $v\in H$. For a subset $X\subseteq  E^0$, the {\it hereditary saturated closure of} $X$, denoted by $\overline{X}$, is the smallest hereditary and saturated subset of $E^0$ containing $X$.

\begin{deff}\label{def:graphmonoid}
Let $E$ be a row-finite graph. We define $F_E$ to be the free commutative monoid generated by $E^0$. The \emph{graph monoid} of $E$, denoted $M_E$, is the 
	quotient of $F_E$ by the relation 
	\[v=\sum_{e\in s^{-1}(v)}r(e)\]
	for every $v\in E^0$ that is not a sink.
\end{deff}

The relations defining $M_E$ can be described more concretely as follows: First, define a relation $\to_1$ on $F_E$ as follows: for $\sum_{i=1}^n v_i  \in F_E$,  and any $1\leq j \leq n$, set 

\[\sum_{i=1}^n v_i \to_1 \sum_{i\not = j }^n v_i+  \sum_{e\in s^{-1}(v_j)}r(e).\]
Then $M_E$ is the quotient of $F_E$ by the {congruence} $\to$ generated by $\to_1$.  To be precise, $\to$ is the smallest reflexive, transitive and additive relation on $F_E$ which contains 
$\to_1$. This relation may be regarded as follows: If $x=\sum_i x_i$ is an element of $F_E$, we may ``let a vertex $x_i$ flow'' to construct the element $y_1=\left(\sum_{j\neq i}x_j\right)+\sum_{e\in s^{-1}(x_i)}r(e)$ with $x\to y_1$. Repeating this procedure and ``letting a vertex of $y_1$ flow'', we construct another element $y_2\in F_E$ such that $y_1\to y_2$. In other words, we simply apply the definition of $\to_1$ to vertices in the representation of elements of $F_E$. By the definition of $\to$, every element $y\in F_E$ such that $x\to y$ may be constructed from $x$ by ``letting its vertices flow successively'' in this manner.

The following lemma is essential to the remainder of this paper, as it allows us to translate the relations in the definition of $M_E$ in terms of the simpler relation $\to$ in $F_E$.

\begin{lemma}[{\cite[Lemmas 4.2 and 4.3]{amp}}]\label{confuu}
For every row-finite graph $E$, the following statements hold:
	
$(1)$	If $a,b\in F_E\setminus\left\{0\right\}$, then $a=b$ in $M_E$ if and only if there exists $c\in F_E$ such that $a\to c$ and $b\to c$. (Note that, in this case, $a=b=c$ in $M_E$.)

$(2)$	If $a=a_1+a_2$ and $a\to b$ in $F_E$, then there exist $b_1,b_2\in F_E$ such that $b=b_1+b_2$, $a_1\to b_1$ and $a_2\to b_2$.
\end{lemma}

Next we define the \emph{talented monoid} $T_E$ of $E$, which is believed to encode the graded structure of the Leavitt path algebra $L_{K}(E)$ (see Conjecture~\ref{conjehfyhtr}) and, later in the paper, plays the role of a bridge between symbolic dynamics and the theory of Leavitt path algebras.

\begin{deff}[{\cite[Page 436]{hazli}}]\label{talentedmon}
For a row-finite graph $E$, the \emph{talented monoid} of $E$, denoted by $T_E$, is the commutative 
	monoid generated by $\{v(i) \mid v\in E^0, i\in \mathbb Z\}$, subject to
	\[v(i)=\sum_{e\in s^{-1}(v)}r(e)(i+1)\]
	for every $i \in \mathbb Z$ and every $v\in E^{0}$ that is not a sink. The additive group $\mathbb{Z}$ of integers acts on $T_E$ via monoid automorphisms by shifting indices: For each $n,i\in\mathbb{Z}$ and $v\in E^0$, define ${}^n v(i)=v(i+n)$, which extends to an action of $\mathbb{Z}$ on $T_E$. Throughout the paper we denote the elements $v(0)$ in $T_E$ by $v$.
\end{deff}

The talented monoid of a graph can also be seen as the graph monoid of the so-called \emph{covering graph} of $E$, denoted by $\overline{E}$: 
we have   $\overline{E}^0=E^0\times\mathbb{Z}$,  $\overline{E}^1=E^1\times\mathbb{Z}$ and the range and source maps are given by
\[s(e,i)=(s(e),i),\qquad r(e,i)=(r(e),i+1).\] Note that the graph monoid $M_{\overline E}$ has a natural $\mathbb Z$-action by ${}^n (v,i)= (v,i+n)$. The following theorem allows us to use Lemma~\ref{confuu} for the talented monoid $T_E$ by identifying it with $M_{\overline E}$. 

\begin{theorem}[{\cite[Lemma 3.2]{hazli}}]\label{hgfgfgggf}
For every row-finite graph $E$,	the correspondence
	\begin{align*}
	T_{E} &\longrightarrow M_{\overline{E}}\\
	v(i) &\longmapsto (v,i)
	\end{align*}
	induces a $\mathbb Z$-monoid isomorphism.
\end{theorem}

We should note that in \cite[Lemma 5.5]{arahazrat} the authors showed that $M_E$ is cancellative if and only if $E$ is acyclic for all $E$. Also, it is not hard to see that $\overline{E}$ is acyclic for all $E$. Therefore, $T_E$ is a cancellative monoid for all $E.$


\subsection{Leavitt path algebras}\label{leviig}
The Leavitt path algebra $L_K(E)$ of a graph $E$ with coefficients in a field $K$ was introduced by Abrams and Aranda Pino in \cite{ap:tlpaoag05}, and independently by Ara, Moreno and Pardo in \cite{amp}. Leavitt path algebras generalize the Leavitt algebras
$L_K(1, n)$ of \cite{leav:tmtoar}, and are intimately related to graph $C^*$-algebras (see \cite{r:ga}). We refer the reader to \cite{a:lpatfd} and \cite{lpabook} for a detailed history and overview of Leavitt path algebras.

\begin{deff}\label{LPAs}
For a row-finite graph $E = (E^0,E^1,s,r)$ and any  field $K$, the \emph{Leavitt path algebra} $L_{K}(E)$ {\it of the graph}~$E$
\emph{with coefficients in}~$K$ is the $K$-algebra generated
by the union of the set $E^0$  and two disjoint copies $E^1$, say $E^1$ and $\{e^*\mid e\in E^1\}$, satisfying the following relations for all $v, w\in E^0$ and $e, f\in E^1$:
	\begin{itemize}
		\item[(1)] $v w = \delta_{v, w} w$;
		\item[(2)] $s(e) e = e = e r(e)$ and $e^*s(e) = e^* = r(e)e^*$;
		\item[(3)] $e^* f = \delta_{e, f} r(e)$;
		\item[(4)] $v= \sum_{e\in s^{-1}(v)}ee^*$ for any  vertex $v$ that is not a sink;
	\end{itemize}
	where $\delta$ is the Kronecker delta.
\end{deff}
It can be shown (\cite[Lemma 1.6]{ap:tlpaoag05}) that $L_K(E)$ is unital if and only if $E^0$ is finite; in this case the identity of $L_K(E)$ is $\sum_{v\in E^0}v$. For any path $p= e_1e_2\cdots e_n$,  the element $e^*_n\cdots e^*_2e^*_1$ of $L_K(E)$ is denoted by $p^*$.   It can be shown (\cite[Lemma 1.7]{ap:tlpaoag05}) that $L_K(E)$ is  spanned as a $K$-vector space by $\{pq^* \mid p, q\in \text{Path}(E), r(p) = r(q)\}$.  Indeed,  $L_K(E)$ is a $\mathbb{Z}$-graded $K$-algebra:  
$L_K(E)= \bigoplus_{n\in \mathbb{Z}}L_K(E)_n$,  where for each $n\in \mathbb{Z}$, the degree $n$ component $L_K(E)_n$ is the set \ $ \text{span}_K \{pq^*\mid p, q\in \text{Path}(E), r(p) = r(q), |p|- |q| = n\}$.

Finding a complete invariant for the classification of Leavitt path algebras is an ongoing endeavor. The Graded Classification Conjecture~(\cite{mathann,hazd}, \cite[\S 7.3.4]{lpabook})  roughly predicts that the graded Grothendieck group $K_0^{\gr}$ classifies
Leavitt path algebras of finite graphs, up to graded isomorphism. The conjecture is closely related to Williams' conjecture (see~\S\ref{dynref}).

In order to state the Graded Classification Conjecture, we first recall the definition of the graded Grothendieck group of a $\Gamma$-graded ring. Given a $\Gamma$-graded ring $A$ with identity and a graded finitely generated projective (right) $A$-module $P$, let $[P]$ denote the class of graded $A$-modules graded isomorphic to $P$. Then the monoid  
\[\mathcal V^{\gr}(A)= \{[P] \mid  P  \text{ is a graded finitely generated projective A-module}\}\]
has a $\Gamma$-module structure defined as follows: for $\gamma \in \Gamma$ and $[P]\in \mathcal V^{\gr}(A)$, $\gamma .[P]=[P(\gamma)]$, where $[P(\gamma)]$ is the $\gamma$-twist of $P$.

The group completion of $\mathcal V^{\gr}(A)$ is called the \emph{graded Grothendieck group} and is denoted by $K^{\gr}_0(A)$.  The $\Gamma$-module structure on $\mathcal V^{\gr}(A)$ induces 
a $\mathbb{Z}[\Gamma]$-module structure on the group $K^{\gr}_0(A)$. In particular, the graded Grothendieck group of a $\mathbb{Z}$-graded ring has a natural $\mathbb{Z}[x,x^{-1}]$-module structure. 

By \cite[Proposition 5.7]{arahazrat}, $T_E$ is $\mathbb{Z}$-monoid isomorphic to the monoid $\mathcal V^{\gr}(L_K(E))$, where $E$ is a row-finite graph and $K$ is an arbitrary field. 

\begin{conj}[{\sc The Graded Classification Conjecture}]\label{conjehfyhtr}
Let $E$ and $F$ be finite graphs, and $K$ a field. Then the following statements are equivalent:
	
$(1)$ The Leavitt path algebras $L_K(E)$ and $L_K(F)$ are graded Morita equivalent;
		
		
$(2)$ There is an order-preserving 
		$\mathbb Z[x,x^{-1}]$-module isomorphism
		$K_0^{\gr}(L(E))\rightarrow K_0^{\gr}(L(F))$;

$(3)$ The talented monoids $T_E$ and $T_F$ are $\mathbb{Z}$-isomorphic.	
\end{conj}
  
We refer the reader to \cite{Abrams-Ruiz-Tomforde2023, arapar, guido,guido1,guidowillie, toke,eilers2, vas} for works on the graded classification conjecture.

\subsection{Symbolic Dynamics}\label{dynref}
The notion of shift equivalence for matrices was introduced by Williams in \cite{williams} (see also~\cite[\S7]{lindmarcus}) in an attempt to provide computable machinery for determining the conjugacy between two shifts of finite type. Recall that two square nonnegative integer matrices $A$ and $B$ are called {\it elementary shift equivalent}, and denoted by $A\sim_{ES} B$, if there are nonnegative matrices $R$ and $S$ such that $A=RS$ and $B=SR$. 
Two square nonnegative integer matrices $A$ and $B$ are called {\it strong shift equivalent}, denoted by $A\sim_{SSE}B$, if there is a sequence of finite elementary shift equivalences from $A$ to $B$. The weaker notion of shift equivalence is defined as follows. The nonnegative integer matrices $A$ and $B$ are called {\it shift equivalent}, denoted by $A\sim_{SE} B$, if there exist $l\ge 1$ and nonnegative matrices $R$ and $S$ such that $A^l=RS$ and $B^l=SR$,  and $AR=RB$ and $SA=BS$. It is not hard to see that  strong
shift equivalence implies shift equivalence (see, e.g., \cite[Theorem 7.3.3]{lindmarcus}). The problem of deciding whether the converse is true remained open for a long time and was known as
Williams' conjecture.

The {\em adjacency matrix} $A_E\in M_{E^0}(\mathbb{N})$ of a  graph $E$ provides the link between symbolic dynamics and  graphs.  By definition, $A_E$ is a square matrix with \[ A_E(v,w) = \left| \left( s^{-1}(v) \cap r^{-1}(w)\right)\right| .\]  Conversely,  any $A \in M_n(\mathbb{N})$ can be interpreted as the adjacency matrix on a graph $E$ with $|E^0| = n$;  $E^1$ consists of precisely $A(i,j)$ edges from vertex $i$ to vertex $j$, for all $1 \leq i, j \leq n$. 

Identifying a square nonnegative integer matrix with its associated graph (and the graph with its adjacency matrix), Williams \cite{williams} showed that two matrices $A$ and $B$ are strongly shift equivalent if and only if one can reach from $A$ to be $B$ with certain graph moves. We recall these in and out-splitting graph moves here, as we will employ them throughout the text. 

\subsection*{Move (I): In-splitting}

\begin{deff}[{\cite[Definition 6.3.20]{lpabook}}]\label{def:insplit}
Let $E$ be a graph. For each $v\in E^0$ with $r^{-1}(v)\neq \emptyset$, take a partition $\left\{\mathscr{E}^v_1,\ldots,\mathscr{E}^v_{m(v)}\right\}$ of $r^{-1}(v)$. We form a new graph $F$ as follows.  Set
	\[F^0=\left\{v_i \mid v\in E^0,1\leq i\leq m(v)\right\}\cup\left\{v \mid r^{-1}(v)=\emptyset \right\}\]
	\[F^1=\left\{e_j \mid e\in E^1,1\leq j\leq m(s(e))\right\}\cup\left\{e \mid r^{-1}(s(e))=\emptyset \right\},\]
	and define the source and range maps  as follows: If $r^{-1}(s(e))\neq\emptyset $, and 
	$e\in\mathscr{E}^{r(e)}_i$, then
	\[s(e_j)=s(e)_j,\qquad r(e_j)=r(e)_i, \text{ for all } 1\leq j \leq m(s(e)).\]
	If $r^{-1}(s(e))=\emptyset$, set $s(e)$ as the original source of $e$, and $r(e)=r(e)_i$, where 
	$e\in\mathscr{E}^{r(e)}_i$.
	
The graph $F$ is called an \emph{in-splitting} of $E$, and conversely $E$ is called an \emph{in-amalgamation} of $F$. We say that $F$ is formed by performing \emph{Move (I)} on $E$.
\end{deff}

\subsection*{Move (O): Out-splitting}

The notions dual to those of in-splitting and in-amalgamation are called \emph{out-splitting} and \emph{out-amalgamation}, respectively. Given a graph $E=(E^0,E^1,r, s)$, the \emph{transpose graph} is defined as $E^*=(E^0,E^1,s,r)$.

\begin{deff}[{\cite[Definition 6.3.23]{lpabook}}]\label{def:outsplit}
	A graph $F$ is an \emph{out-splitting} (\emph{out-amalgamation}) of a graph $E$ if $F^*$ is an in-splitting (in-amalgamation) of $E^*$, and we say that $F$ is formed by performing \emph{Move (O)} on $E$.
	
\end{deff}

We emphasize that in this paper, we require the sets $\mathscr E^i_v$ used in in- or out-splitting to be non-empty.


Let $E$ be an essential graph with the discrete topology on $E^1$. Define the topological edge shift $(X_E, \sigma_E)$ by setting:
$$X_E:= \{x=(x_n)_{n\in \mathbb{Z}}\mid x_n\in E^1 \text{ such that } s(x_n) = r(x_{n+1})\}\subseteq (E^1)^{\mathbb{Z}}$$
where $\sigma_E: X_E\longrightarrow X_E$ is the shift map with $\sigma(x)_n = x_{n+1}$ for all $n\in \mathbb{Z}$.\medskip

Two edge shifts $X_E$ and $X_F$ is {\it conjugate}, denoted $X_E\cong X_F$, if there is a homeomorphism $f: X_E\longrightarrow X_F$ such that the following diagram
$$\xymatrix{X_E\ar[d]_{\sigma_E} \ar[r]^f& X_F \ar[d]^{\sigma_F}\\
	X_E \ar[r]^f & X_F}$$
is commutative.

We are now able to present the precise form of Williams' celebrated criteria for the conjugacy of edge shifts  established in \cite{williams}.

\begin{theorem}[{\sc Williams}~\cite{williams}]\label{willimove}
Let $E$ and $F$ be essential graphs, and let $A_E$ and $A_F$ be the adjacency matrices of $E$ and $F$, respectively. Then, the following statements are equivalent:\medskip

$(1)$ $X_E$ is conjugate to $X_F;$\medskip

$(2)$ $A_E\sim_{SSE}A_F$;\medskip

$(3)$ $E$ can be obtained from $F$ by a sequence of in-splittings, out-splittings,  in-amalgamations, and out-amalgamations.
\end{theorem}

It is worth mentioning the following example.

\begin{example}[{\cite[Example 7.3.13]{lindmarcus}}]\label{SE-SSE-Exam}
Let $E$ and $E$ be the following graphs:
\begin{center}
$E=\ \ \ \ \ \xymatrix{\bullet^v\ar@(ld,lu)\ar@/^.5pc/[r]^{k\ edges}        &\bullet^w \ar@(ru,rd)\ar@/^.5pc/[l]^{ (k-1)\ edges}}$ \quad\quad  and \quad\quad
	$F=\ \ \ \ \ \xymatrix{\bullet^v\ar@(ld,lu)\ar@/^.5pc/[r]^{k(k-1)\ edges}        &\bullet^w\ar@(ru,rd) \ar@/^.5pc/[l]}$
\end{center}	
We have that the adjacency matrices of $E$ and $F$ are respectively the following matrices: 
\begin{center}
$A_E =\left(
\begin{array}{cccc}
1 &k\\
k-1 &1 \\
\end{array}
\right)$\quad\quad and \quad\quad $A_F =\left(
\begin{array}{cccc}
1 &k(k-1)\\
1 &1 \\
\end{array}
\right).$ 	
\end{center}
Then $A_E \sim_{SE} A_F$ for all $k\ge 1$ and $A_E \sim_{SSE} A_F$ for all $1\le k\le 3$, but it is not known whether $A_E \sim_{SSE} A_F$ when $k\ge 4$.
\end{example}

Considering how difficult this is, anything new we can say about
strong shift equivalence while merely assuming shift equivalence is bound to be interesting. The following conjecture was proposed by Williams \cite{willwrong} in $1974$.

\begin{conj}[{\sc Williams' conjecture \cite{willwrong}}]\label{Wilconj}
$A_E\sim_{SE}A_F\ \Longleftrightarrow\ A_E\sim_{SSE}A_F$ for essential graphs $E$ and $F$.
\end{conj}

It took 25 years before a counterexample  was found by Kim and Roush \cite{kimroush99}. More precisely, Kim and Roush showed that there are essential graphs $E$ and $F$ of order $7$ with $A_E\sim_{SE} A_F$ while $A_E\nsim_{SSE} A_F$. However, identifying classes of edge shifts for which shift equivalence and strong shift equivalence are equivalent is an open question.

Although strong shift equivalence characterizes conjugacy of shifts of finite type, there is no general algorithm for deciding whether $X_A$ and $X_B$ are strongly shift equivalent, even for $2 \times 2$ matrices $A,B.$ The weaker notion of shift equivalence  is easier to analyze, as Krieger established.  In~\cite{krieger}, he 
proposed an invariant for classifying the irreducible shifts of finite type up to shift equivalence. Surprisingly, Krieger's dimension group in symbolic dynamics turns out to be expressible as the graded Grothendieck group of a  Leavitt path algebra. 

Let $A$ be a nonnegative integral $n\times n$-matrix. Consider the sequence of free ordered abelian groups $\mathbb Z^n \to \mathbb Z^n \to \cdots$, where the ordering in $\mathbb Z^n$ is defined point-wise (i.e.,  the positive cone is $\mathbb N^n$).  Then  $A$ acts as an order-preserving group homomorphism on this sequence, as follows: 
\[\mathbb Z^n \stackrel{A}{\longrightarrow} \mathbb Z^n \stackrel{A}{\longrightarrow}  \mathbb Z^n \stackrel{A}{\longrightarrow} \cdots.
\]
The direct limit of this system, $\Delta_A:= \varinjlim_{A} \mathbb Z^n$, along with its positive cone, $\Delta^+_A$, and the automorphism which is induced by $A$ on the direct limit, 
$\delta_A:\Delta_A \rightarrow \Delta_A$, is the invariant considered by Krieger, now known as  \emph{Krieger's dimension group}.  Following~\cite{lindmarcus}, we denote this triple by $(\Delta_A, \Delta_A^+,  \delta_A)$. 

The following theorem was proved by Krieger (\cite[Theorem~4.2]{krieger}, 
see also~\cite[\S7.5]{lindmarcus} for a detailed algebraic treatment). 

\begin{theorem}[{\cite[Theorem~4.2]{krieger}}]\label{kriegeror}
	Let $A$ and $B$ be two square nonnegative integer matrices. Then, $A$ and $B$ are shift equivalent if and only if 
	\[(\Delta_A, \Delta_A^+, \delta_A) \cong (\Delta_B, \Delta_B^+, \delta_B).\]
\end{theorem}

Wagoner noted that the induced structure on $\Delta_A$ by the automorphism $\delta_A$ makes $\Delta_A$ a $\mathbb Z[x,x^{-1}]$-module, where the action of $\delta_A$ is  multiplication by $x^{-1}$ (see \cite[Lemma 11]{hazli}).  This fact  was systematically used in~\cite{wago1,wago2} (see also~\cite[\S3]{boyle}).

Let $E$ be a row-finite graph and $K$ a field. Recall that the graded Grothendieck group $K_0^{\gr}(L_K(E))$  has a natural $\mathbb Z[x,x^{-1}]$-module structure.  Theorem~\ref{mmpags} shows that the graded Grothendieck group of the Leavitt path algebra $L_K(E)$  coincides with the Krieger's dimension group associated to $A_E^t$, the transpose of the adjacency matrix $A_E$: 
\[\big(K_0^{\gr}(L_K(E)),(K_0^{\gr}(L_K(E))^+\big ) \cong (\Delta_{A^t_E}, \Delta_{A^t_E}^+).\] 
This will provide a link between the theory of Leavitt path algebras and symbolic dynamics~\cite{arapar,hazd,hazbk}.

\begin{theorem}[{\cite{arapar, hazli}}]\label{mmpags}
Let $E$ be an essential graph with adjacency matrix $A_E$ and $K$ a field. Then, there is a $\mathbb Z[x,x^{-1}]$-module isomorphism 
	$\phi:K^{\gr}_0(L_K(E)) \longrightarrow \Delta_{A^t_E}$ such that $\phi(K^{\gr}_0(L_K(E))^+)=\Delta_{A^t_E}^+$. 
\end{theorem}

It is easy to see that two matrices $A$ and $B$ are shift equivalent if and only if so are $A^t$ and $B^t$. Combining this with Theorems~\ref{kriegeror} and \ref{mmpags}, we have the following corollary, which will be used to prove our main result (Theorem \ref{mainthm}).

\begin{corollary}[{\cite[Corollary 12]{hazli}}]\label{h99}
Let $E$ and $F$ be essential graphs, and $A_E$ and $A_F$ their adjacency matrices, respectively. Let $K$ be an arbitrary field. Then
$A_E$ is shift equivalent to $A_F$ if and only if $K_0^{\gr}(L_K(E)) \cong K_0^{\gr}(L_K(F))$ via an order-preserving $\mathbb Z[x,x^{-1}]$-module isomorphism.
\end{corollary}

\section{The dynamics of graphs of Gelfand-Kirillov dimension 3} \label{sec3}
In this section, we show that the class of graphs with disjoint cycles is closed under shift equivalence and shift equivalence preserves the cycle structure of these graphs (Theorem \ref{corcycle1}). Also, we provide number-theoretic criteria for graphs of Gelfand-Kirillov dimension three to be strong shift equivalent (Theorems \ref{numtheo-crite-sse1} and \ref{numtheo-crite-sse2}).\medskip


Let $E$ be a row-finite graph. We define a  preorder $\le$ on $E^0$ given by:
\begin{center}
$w\le v$ in the case there exists a path $p\in \text{Path}(E)$ such that $s(p) = v$ and $r(p) = w$.	
\end{center}
(We will sometimes equivalently write $v\ge w$ in this situation.) 
For any cycle $c$ in $E$ and any vertex $w\in E^0$, we write $w \le c$ in the case there is a path $p\in \text{Path}(E)$ such that $r(p) =w$ and $s(p) = v$ for some $v\in c^0$.

We denote by $\mathcal{C}_E$ the set of all distinct cycles in $E$.  We define a  preorder $\le$ on $\mathcal{C}_E$ given by:
\begin{center}
	$c\le c'$ in the case $w \le c'$ for some $v\in c^0$
\end{center}
for all $c, c'\in \mathcal{C}_E$.

The graph $E$ is called a {\it graph with disjoint cycles} if every vertex in $E$ is the base of at most one cycle. We note that $(\mathcal{C}_E, \le)$ is a partially ordered set for every graph with disjoint cycles $E$.

\begin{prop}\label{corcycle}
Let $E$ be a finite graph with disjoint cycles. Then there is a one-to-one  order preserving correspondence between cycles in the graph $E$ and $\mathbb Z$-order ideals $I$ of the talented monoid $T_E$ with unique sub-ideal $J$ such that $I/J$ is simple $n$-periodic, where $n$ is the length of the corresponding cycle. 	
\end{prop}
\begin{proof}
Let $c$ be a cycle in $E$ and $z\in c^0$. We show that the $\mathbb Z$-order ideal $\langle z \rangle$ of $T_E$ generated by $z$ has a unique sub-ideal with the quotient being simple $|C|$-periodic monoid. 

We first note that $\langle z \rangle = \{ x \in T_E \mid x \leq \sum_{i\in \mathbb Z} {}^{i} z \}= \langle \overline{\{z\}}\rangle $. Consider the sub-ideal $I_c$ of  $\langle z \rangle$, generated by the set 
\[S=\big \{v\in E^0 \mid c \ge v \text{ and } v\not \in c^0 \big \}.\] (Note that $S$ could be an empty set as well.)
Since for any $v\in S$, we have $z\geq v$ (i.e., there is a path from $z$ to $v$), and so $v\in \langle z \rangle$. It follows that $I_c \subseteq \langle z \rangle$. We show that $I_c$ is the unique $\mathbb Z$-sub-ideal of $\langle z \rangle$ such that the quotient is simple and periodic. We have  $I_c \not = \langle z \rangle$. For if $z\in I_c$, then since $I_c = \{x\in T_E\mid x \le \sum_{i\in \mathbb Z, v\in S} {}^i v\}$, it follows that for some $t\in T_E$, \[z+t=\sum_{i\in \mathbb Z, v\in S} {}^i v.\] Since $z\in c^0$, any transformation of $z$ produces a vertex on the cycle $c$. Now Lemma~\ref{confuu} gives that both sides of the equality transform to sum $l$ of some  vertices.  Combing these, it shows that there is a path from $v$ to the cycle $c$, which means that  $E$ is not a graph with disjoint cycles, a contradiction. Therefore, $I_c$ is a proper subset of $\langle z \rangle$. Next we consider the monoid $\langle z \rangle / I_c$. Since $E$ is a graph with disjoint cycles and $z\in c^0$, the transformation 
	\begin{equation}\label{noalgdhd}
	z\rightarrow_1 w(1)+l,
	\end{equation}
	gives a vertex $w\in c^0$ and $l$ which is a sum of vertices from $S$ (with suitable shifts). Repeating this $n=|C|$ times, we get that $z= z(n)+l'$, where $l'\in I_c$. This shows $[z]={}^n[z]$ in the quotient monoid $\langle z \rangle / I_c$.  Next we show that any element of $\langle z \rangle / I_c$ can be represented as a sum $\sum_{i\in \mathbb Z} z(i)$. Let $x\in \langle z \rangle$. Thus 
	\[x+t=\sum_{i\in \mathbb Z} {}^i z,\] for some $t\in T_E$. By Lemma~\ref{confuu} (1), $x+t \rightarrow l$ and $\sum_{i\in \mathbb Z} {}^i z \rightarrow l$. By Lemma~\ref{confuu} (2), we can write $l=l_1+l_2$, where $x\rightarrow l_1$. A similar argument as in (\ref{noalgdhd}), gives that $l_1= \sum_{j\in \mathbb Z} {}^j z + l'$, where $l'\in I_c$. Thus $[x]=\sum_{j\in \mathbb Z, v\in S} [z(j)]$ in the quotient monoid. Since each of $z(j)$ is $n$-periodic in $\langle z \rangle / I_c$, we see that $x$ is $n$-periodic. Since $[z]$ generates the monoid $\langle z \rangle / I_c$,  the quotient is simple as well. 
	
	Next we show that $I_c$ is the unique ideal with this property in $\langle z \rangle$. Suppose $J\subset \langle z \rangle$ is a proper sub-ideal such that $\langle z \rangle / J$ is also simple and $n$-periodic. Since $z\not \in J$, $[z]$ generates the $\mathbb Z$-monoid $\langle z \rangle / J$. Let $v\in S$ and consider the element $[v] \in \langle z \rangle / J$. If $[v]=0$ then $v\in J$. Otherwise, since $[z]$ is a generator, in $T_E$ we have 
	\[v+j_1=\sum_{i\in \mathbb Z} z(i) +j_2,\] for some $j_1 , j_2 \in J$. Once again a similar argument with Lemma~\ref{confuu}, shows that either there is a path from $v$ to $c$ (a contradiction) or $z$ appears in some  sum of elements of $J$, i.e, $z\in J$ again a contradiction. Thus for any $v\in S$, $[v]=0$, i.e., 
	$S\subseteq J$ and since $I_c$ is a proper maximal ideal, we have $J=I_c$. 
	
	Conversely, suppose $I\subseteq T_E$ is an $\mathbb Z$-order ideal  with unique sub-ideal $J$ such that $I/J$ is simple $n$-periodic. Then $I$ determines a hereditary saturated subset $H\subseteq E^0$ and $J$ a hereditary saturated subset $K\subseteq H$. We have $I\cong T_H$ and $J\cong T_K$ and $I/J\cong T_H/T_K\cong T_{H/K}$. Since $I/J$ is simple $n$-periodic, then $H/K$ is a comet graph with a unique cycle $c$ of length $n$. Since $J$ is unique, without ambiguity we assign this cycle $c$ to $I$. 
	
	Next we show these assignments give a one-to-one correspondence between the prescribed ideals in $T_E$ and cycles in $E$. 
	
	Starting from a cycle $c$, and choosing a vertex $z\in c^0$, we obtain the ideal $\langle z \rangle$ which by has a unique ideal $I_c$ with simple $|c|$-periodic quotient $\langle z \rangle/I_c$. From the construction of the unique ideal $I_c$ above, the quotient $\langle z \rangle/I_c$ determines $c$. For the converse, first note that if $H$ is a hereditary and saturated subset of $E$, then $$H=\overline{\Big \{\bigcup c^0\mid c \text{ a cycle or a sink in } H\Big \}}.$$ For any maximal cycle $d$ in $H$, let $$K_d = \overline{\Big\{\bigcup c^0\mid c \text{ a cycle or a sink in } H, c\not = d\Big\}}.$$ We then have $d\not \subseteq K_d$ and $H/K_d$ is a comet graph with $d$ as a unique cycle. Let $I$ be an ideal of $T_E$ such that there is a unique ideal $J$ such that $I/J$ is periodic. By previous argument, we see that the hereditary saturated subset $H$ correspond to $I$ has only one maximal cycle $c$. Choosing a vertex $z\in c^0$, clearly $I=\langle z \rangle$. This gives the converse of the argument.

	Finally it is immediate from the construction that the above correspondences are order preserving, thus finishing the proof.
\end{proof}

Let $M$ be a $\Gamma$-monoid. We say that $M$ has a {\it composition series} if there exists a chain of $\Gamma$-order ideals
\[0 = I_0 \subset I_1\subset I_2\subset \cdots \subset I_n = M,\] such that $I_{i+1}/I_i$, $0\le i\le n-1$, are simple $\Gamma$-monoids.
We say a composition series is of the {\it cyclic type} if all the simple quotients  $I_{i+1}/I_i$ are cyclic. In particular, in \cite[Lemma 4.1]{alfi} the authors showed that for every finite graph $E$, $T_E$ has a cyclic composition series if and only if $E$ is a graph with disjoint cycles and no sinks.

The following theorem provides that shift equivalence preserves the cycle structure of graphs with disjoint cycles.

\begin{thm}\label{corcycle1}
Let $E$ and $F$ be finite essential graphs such that $A_E\sim_{SE} A_F$. Then, $E$ is a graph with disjoint cycles if and only if so is $F$.  In this case, there exists an order-preserving bijection $\phi: \mathcal{C}_E\longrightarrow\mathcal{C}_F$ such that $|c| = |\phi(c)|$ for all $c\in \mathcal{C}_E$.
\end{thm}

\begin{proof}
Since $A_E\sim_{SE} A_F$ and by Corollary \ref{h99}, there is a $\mathbb Z$-isomorphism $T_E\cong T_F$. By \cite[Lemma 4.1]{alfi}, $T_E$ has a cyclic composition series if and only if $E$ is a graph with disjoint cycles. Then, since $T_E\cong T_F$ as $\mathbb{Z}$-monoids, $E$ is a graph with disjoint cycles if and only if so is $F$. The last statement immediately follows from Proposition~\ref{corcycle} and the fact that for any finite graph $G$, $T_G$ is simple and $n$-periodic if and only if $G$ is a graph having only one cycle $c$ of length $n$ such that all vertices are connected to some vertex in $c^0$ (by \cite[Corollary 3.12]{alfi}), thus finishing the proof.
\end{proof}


Given a field $K$ and a finitely generated $K$-algebra $A$. The \textit{Gelfand-Kirillov dimension} of $A$ ($\text{GKdim(A)}$ for short) is defined to be
$$\text{GKdim}(A) :=\limsup\limits_{n\rightarrow \infty}\log_{n}(\dim(V^{n})),$$
where $V$ is a finite dimensional subspace of $A$ that generates $A$ as an algebra over $K$. This definition is independent of the choice of $V$. 

In \cite{aajz:lpaofgkd} Alahmadi, Alsulami, Jain and Zelmanov  determined the Gelfand-Kirillov dimension of Leavitt path algebras of finite graphs. We should mention this result here. To do so, we need to recall useful notions of graph theory.

Let $E$ be a finite graph. For two cycles $c$ and $c'$ in $E$, we write $c\Rightarrow c'$ if there exists a path that starts in $c$ and ends in $c'$. A sequence of cycles $c_1, \cdots , c_k$ is a chain of length $k$ if $c_1\Rightarrow \cdots \Rightarrow c_k$. We say that such a chain has an {\it exit} if the cycle $c_k$ has an exit. Let $d_1$ be the maximal length of a chain of cycles in $E$, and let $d_2$ be the maximal length of a chain of cycles with an exit in $E$. For every field $K$, by \cite[Theorem 5]{aajz:lpaofgkd}, the Gelfand-Kirillov dimension of $L_K(E)$ is finite if and only if $E$ is a graph with disjoint cycles. In this case,  $\text{GKdim}(L_K(E)) =\max\{2d_1-1, 2d_2\}$. In particular, if, in addition, $E$ is essential, then $\text{GKdim}(L_K(E))$ is an odd number.	Moreover, the Gelfand-Kirillov dimension of $L_K(E)$ does not depend on the base field $K$. Motivated by these observations, we have the following definition.

\begin{deff}\label{graphwithfGK}
A {\it graph of finite Gelfand-Kirillov dimension} is a connected finite essential graph such that its Leavitt path algebra has finite Gelfand-Kirillov dimension.
\end{deff}

We note that a graph of Gelfand-Kirillov dimension one is a unique cycle. The following remark provides us with a complete description of graphs of Gelfand-Kirillov dimension three.

\begin{remdef}\label{rem-gwGK3}
A finite graph $E$ is a graph of Gelfand-Kirillov dimension $3$ if and only if $E$ is a connected essential graph consisting of $m+n$ disjoint cycles $C_1^E,\ldots , C_m^E$, $D_1^E,\ldots, D_n^E$ and the paths connecting from $C_i^E$ to $D_j^E$ for all $1\le i\le m$ and $1\le j\le n$. 

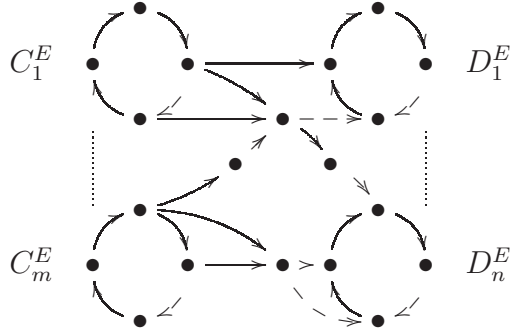
$$\xymatrixrowsep{0.4pc}\xymatrixcolsep{0.4pc}\xymatrix{&&\bullet\ar@/^0.5pc/[rd]&&&&&\bullet\ar@/^0.5pc/[rd]\\
	C^E_1&\bullet \ar@/^0.5pc/[ru]&&\bullet\ar@/^0.5pc/@{-->}[ld]\ar[rrr]\ar@/^.2pc/[rrd]&&&\bullet\ar@/^0.5pc/[ru]&&\bullet\ar@/^0.5pc/@{-->}[ld]&D^E_1\\
	&\ar@{..}[dd]&\bullet\ar@/^0.5pc/[lu]\ar[rrr]&&&\bullet\ar@/^.2pc/[rd]\ar@{-->}[rr]&&\bullet\ar@/^0.5pc/[lu]&\ar@{..}[dd]\\	&&&&\bullet\ar[ru]&&\bullet\ar@/^0.2pc/@{-->}[rd]&\\
	&&\bullet\ar@/^0.5pc/[rd]\ar@/_0.2pc/[urr]\ar@/^.5pc/[rrrd]&&&&&\bullet\ar@/^0.5pc/[rd]&\\
	C^E_m&\bullet\ar@/^0.5pc/[ru]&&\bullet\ar@/^0.5pc/@{-->}[ld]\ar[rr]&&\bullet\ar@/_.7pc/@{-->}[rrd]\ar@{-->}[r]&\bullet\ar@/^0.5pc/[ru]&&\bullet\ar@/^0.5pc/@{-->}[ld]&	D^E_n\\
	&&\bullet\ar@/^0.5pc/[lu]&&&&&\bullet\ar@/^0.5pc/[lu]}$$
\end{remdef}

For every graph of Gelfand-Kirillov dimension $3$,  we call $C^E_i$'s the {\it source cycles}, and $D^E_j$'s the {\it sink cycles}. 

A path $p=e_1\cdots e_t$ in $E$ is called a {\it trail} if $s(p)\in (C^E_i)^0$ for some $1\le i\le m$,  $r(p)\in (D^E_j)^0$ for some $1\le j\le n$, and $r(e_k)\notin \bigcup^n_{j=1}\bigcup^m_{i=1} (C^E_i)^0 \cup (D^E_j)^0$ for all $1\le k\le t$.

For every trail $p$, its {\it interior} is the set $p^0\setminus (\bigcup^n_{j=1}\bigcup^m_{i=1} (C^E_i)^0 \cup (D^E_j)^0)$. A vertex $v$ in $E$ is called an {\it interior vertex} if $v$ is in the interior of some trail. 

\begin{deff}\label{norform}
A graph $E$ of Gelfand-Kirillov dimension three is in \textit{normal form} if all trails of $E$ consist of exactly one edge, and for each source cycle $C^E$, all trails of $E$ starting from $C^E$ have the same source.
$$\xymatrixrowsep{0.4pc}\xymatrixcolsep{0.4pc}
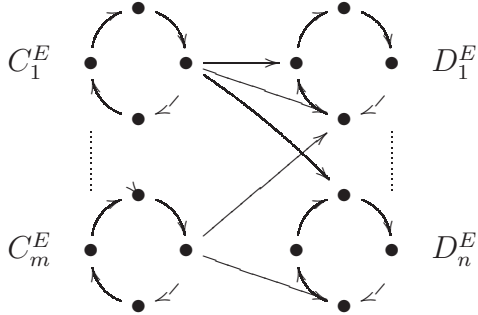
\xymatrix{&&\bullet\ar@/^0.5pc/[rd]&&&&&\bullet\ar@/^0.5pc/[rd]\\
	C^E_1&\bullet\ar@/^0.5pc/[ru]&&\bullet\ar@/^0.5pc/@{-->}[ld]\ar[rrr]\ar[rrrrd]\ar@/^.2pc/[rrrrddd]&&&\bullet\ar@/^0.5pc/[ru]&&\bullet\ar@/^0.5pc/@{-->}[ld]&D^E_1\\
	&\ar@{..}[dd]&\bullet\ar@/^0.5pc/[lu]&&&&&\bullet\ar@/^0.5pc/[lu]&\ar@{..}[dd]\\ &&&&&&\\
	&&\bullet\ar@/^0.5pc/[rd]\ar[]&&&&&\bullet\ar@/^0.5pc/[rd]&\\
	C^E_m&\bullet\ar@/^0.5pc/[ru]&&\bullet\ar@/^0.5pc/@{-->}[ld]\ar[rrrruuu]\ar[rrrrd]&&&\bullet\ar@/^0.5pc/[ru]&&\bullet\ar@/^0.5pc/@{-->}[ld]&D^E_n\\
	&&\bullet\ar@/^0.5pc/[lu]&&&&&\bullet\ar@/^0.5pc/[lu]}$$
\end{deff}

The following proposition gives that every graph of Gelfand-Kirillov dimension three can be transformed into a graph in normal form.

\begin{prop}\label{norform-prop}
Every graph of Gelfand-Kirillov dimension three can be transformed into a graph of Gelfand-Kirillov dimension three in normal form by a sequence of in-splittings, out-splittings,  in-amalgamations, and out-amalgamations.
\end{prop}
\begin{proof}
Let $E$ be a graph of Gelfand-Kirillov dimension three, and $v$ an interior vertex in $E$.	We define the following two distances:
\[d(v) = \max\{|p|\mid p \text{ is a trail in $E$ starting at $v$ and ending at some sink cycle}\}\] and 
\[\ell(v) = \max\{|p|\mid p \text{ is a trail in $E$ starting at some source cycle and ending at $v$}\}.\]	

We first transform $E$ into another graph of Gelfand-Kirillov dimension three where the interiors of distinct trails do not intersect. Indeed, assume there is some vertex u in the interior of a trail which is the source of more than one edge. Assuming the prior case, let us temporarily call such a vertex a {\it multi-source vertex}. Performing a sequence of out-splitting at these multi-source vertices $v$ via a partition $\{\mathscr{E}^v_1, \hdots, \mathscr{E}^v_k\}$ of $s^{-1}(v) = \{e_1, \hdots, e_k\}$ where $\mathscr{E}^v_i = \{e_i\}$, we obtain new vertices $v_1, v_2, \ldots, v_k$ which are the sources of only one edge each. The multi-source vertices in this new graph are the same ones as in the original graph, except for $v,$ plus all vertices in $s(r^{-1}(v))$ are now multi-source vertices. We note that $d(x) > d(v)$ for all $x\in s(r^{-1}(v))$. Thus, performing out-splittings at all multi-sources, starting arbitrarily at the ones with the smallest $d$-distance and moving to the ones with higher $d$-distance, gets rid of all of them. In this manner, we may transform our original graph into one in which all the vertices in the interior of the trails are only sources of one edge each. 

$$\xymatrixrowsep{0.9pc}\xymatrixcolsep{0.9pc} \xymatrix{\bullet_{u_1}\ar@/^.5pc/[dr]&&\bullet\ar@/^.5pc/[rr]&&\bullet\ar@/^1.3pc/[rr]&&\bullet\ar@/^0.6pc/[dl]&&&&&\bullet_{u_1}\ar@/_.6pc/[ddrr]\ar@/^.5pc/[drr]&&&\bullet\ar@/^0.3
	pc/[rr]&&\bullet\ar@/^1.3pc/[rr]&&\bullet\ar@/^0.6pc/[dl]\\
	&\bullet_{v}\ar@/^.4pc/[ur]\ar@/^1pc/[drr]&&&&\bullet\ar@/^0.6pc/[ul]&&\ar[rrr]^{\text{out-split at $v$}}&&&&&&\bullet_{v_1}\ar@/^.4pc/[ur]&&&&\bullet\ar@/^0.6pc/[ul]\\
	\bullet_{u_2}\ar@/_.4pc/[ur]&&&\bullet\ar@/_.8pc/[r]&\bullet\ar@/_.8pc/[l]&&&&&&&\bullet_{u_2}\ar@/_.5pc/[rr]\ar@/_.5pc/[urr]&&\bullet_{v_2}\ar@/^.5pc/[rr]&&\bullet\ar@/_.8pc/[r]&\bullet\ar@/_.8pc/[l]\\
	\bullet_{u_3}\ar@/_.6pc/[uur]&&&&&&&&&&&\bullet_{u_3}\ar@/_.5pc/[uurr]\ar@/_.5pc/[urr]}$$

\begin{figure}[h]
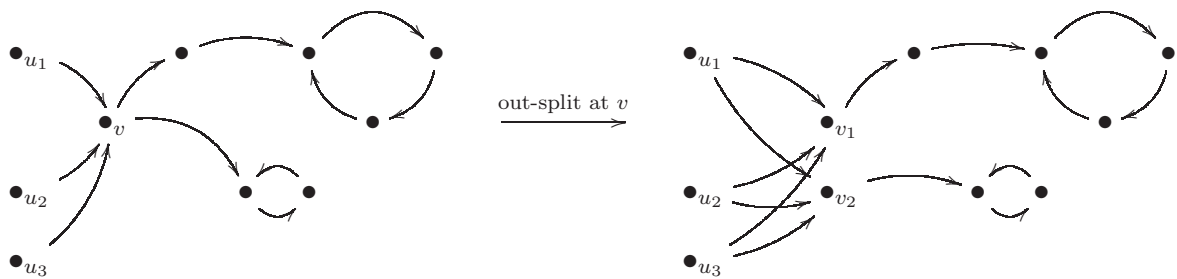
\caption{out-splitting at $v$ makes new multi-source vertices $u_1, u_2, u_3$  instead of $v$ which $d(u_1),d(u_2),d(u_3) > d(v)$.}
\end{figure}

By the same argument, assume there is some vertex $v$ in the interior of a trail which is the range of more than one edge. Assuming the prior case, let us temporarily call such a vertex a {\it multi-range vertex}. Performing a sequence of in-splittings at these multi-range vertices $v$ via a partition $\{\mathscr{E}^v_1, \hdots, \mathscr{E}^v_h\}$ of $r^{-1}(v) = \{e_1, \hdots, e_h\}$ where $\mathscr{E}^v_i = \{e_i\}$, we obtain new vertices $v_1, v_2 ,\ldots, v_h$ which are the range of only one edge each. The multi-range vertices in this new graph are the same ones as in the original graph, except for $u,$ plus all vertices in $r(s^{-1}(v))$ are now  multi-range vertices. We also note that $\ell(x) > \ell(v)$ for all $x\in r(s^{-1}(v))$. Thus, performing in-splittings at all multi-ranges, starting arbitrarily at the ones with the smallest $\ell$-distance and moving to the ones with higher $\ell$-distance, gets rid of all of them. In this manner, we may transform our original graph into one in which all vertices in the interior of the trails are only ranges and sources of one edge each.

$$\xymatrixrowsep{0.9pc}\xymatrixcolsep{0.9pc} 
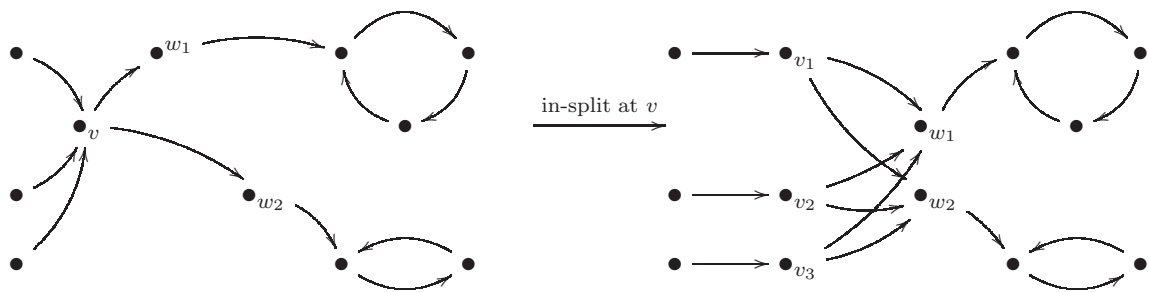
\xymatrix{\bullet\ar@/^.5pc/[dr]&&\bullet^{w_1}\ar@/^.5pc/[rr]&&\bullet\ar@/^1.3pc/[rr]&&\bullet\ar@/^0.6pc/[dl]&&&&\bullet\ar[rr]&&\bullet_{v_1}\ar@/_.6pc/[ddrr]\ar@/^.5pc/[drr]&&&\bullet\ar@/^1.3pc/[rr]&&\bullet\ar@/^0.6pc/[dl]\\
	&\bullet_{v}\ar@/^.4pc/[ur]\ar@/^0.5pc/[drr]&&&&\bullet\ar@/^0.6pc/[ul]&&\ar[rrr]^{\text{in-split at $v$}}&&&&&&&\bullet_{w_1}\ar@/^.4pc/[ur]&&\bullet\ar@/^0.6pc/[ul]\\
	\bullet\ar@/_.4pc/[ur]&&&\bullet_{w_2}\ar@/^.5pc/[dr]&&&&&&&\bullet\ar[rr]&&\bullet_{v_2}\ar@/_.5pc/[rr]\ar@/_.3pc/[urr]&&\bullet_{w_2}\ar@/^.3pc/[dr]&&&\\
	\bullet\ar@/_.6pc/[uur]&&&&\bullet\ar@/_.8pc/[rr]&&\bullet\ar@/_.8pc/[ll]&&&&\bullet\ar[rr]&&\bullet_{v_3}\ar@/_.5pc/[uurr]\ar@/_.5pc/[urr]&&&\bullet\ar@/_.8pc/[rr]&&\bullet\ar@/_.8pc/[ll]}$$

\begin{figure}[h]\caption{out-splitting at $v$ makes  new multi-range vertices $w_1, w_2$  instead of $v$ which $\ell(w_1), \ell(w_2) > \ell(v)$.}
\end{figure} 

Now, let $C$ be an arbitrary source cycle in $E$ and $v$ a vertex on $C$ which is the source of a trail $p = e_1\cdots e_t$. Performing the out-splitting at $v$ with a partition $ \{\mathscr{E}^v_1, \mathscr{E}^v_2\} $ of $s^{-1}(v)$, where  $\mathscr{E}^v_1 = s^{-1}(v)\setminus \{e_1\}$ and  $\mathscr{E}^v_2 = \{e_1\} .$ Doing this split $v$ to $v_1,v_2$ , where $v_1$ will replace $v$ as a vertex in $C$ and lengthen $p$ and move  $p$ source back one step along $C$. By doing this repeatedly, we can arrange so that all the trails starting from $C$ are sourced at the same vertex.

$\xymatrixrowsep{0.8pc}\xymatrixcolsep{1pc}\xymatrix{&&&&\ar@{..}[r]&&&&&&&&&&&&\ar@{..}[r]&\\\ar@/_2.1pc/@{..}[dddd]&&\bullet\ar@/_.45pc/[ll]\ar@{-->}[urr]&&&&&&&&&&\ar@/_2.1pc/@{..}[dddd]&&\bullet\ar@/_.45pc/[ll]\ar@{-->}[urr]\\&&&&&\ar@{..}[r]&&&&&&&&&&&&\ar@{..}[r]&\\&&&\bullet_v\ar@/_.5pc/[uul]\ar[drr]\ar@{-->}[urr]&&&&\ar[rrr]^{\text{ out-split at $v$ }}&&&&&&&&\bullet_{v_1}\ar@/_.45pc/[uul]\ar@{-->}[urr]&&\bullet_{v_2}\ar[dr]\\&&&&&\ar@{..}[r]&&&&&&&&&&&&&\ar@{..}[r]&\\\ar@/_.45pc/[rr]&&\bullet\ar@/_.5pc/[uur]\ar@{-->}[drr]&&&&&&&&&&\ar@/_.45pc/[rr]&&\bullet\ar@/_.45pc/[uur]\ar[uurrr]\ar@{-->}[drr]&&&&\\&&&&\ar@{..}[r]&&&&&&&&&&&&\ar@{..}[r]&}$
\begin{figure}[h]
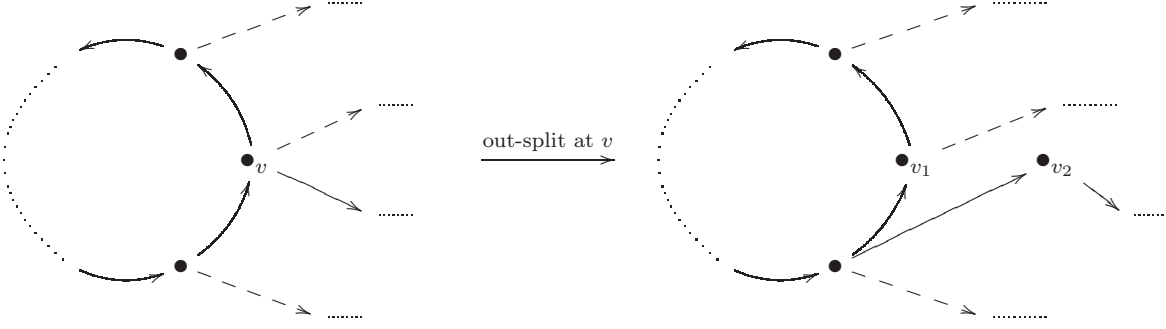
\caption{Out-splitting at a vertex on a source cycle. The dashed arrows represent trails that may or may not exist, but nevertheless are not modified by the graph move.}
\end{figure} 

Finally, let $D$ be an arbitrary sink cycle in $E$ and $v$ a vertex in $D$ which is the range of a trail $p = e_1\cdots e_t$ of length more than one. Performing the in-amalgamation at $v$ with a partition $ \{\mathscr{E}^v_1, \mathscr{E}^v_2\} $ of $s^{-1}(v)$, where  $\mathscr{E}^v_1 = r^{-1}(v)\setminus \{e_t\}$ and  $\mathscr{E}^v_2 = \{e_t\}$, we can shorten the trail $p$ at the cost of moving their ranges back along $D$. 
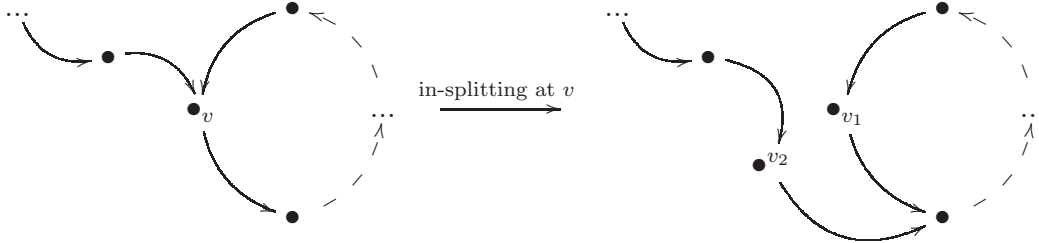
$$\xymatrixrowsep{0.5pc}\xymatrixcolsep{0.5pc}\xymatrix{...\ar@/_0.8pc/[drr]&&&&&&\bullet\ar@/_0.8pc/[ddll]&&&&&&&&...\ar@/_0.8pc/[drr]&&&&&&\bullet\ar@/_0.8pc/[ddll]&&\\&&\bullet\ar@/^0.8pc/[drr]&&&&&&&&&&&&&&\bullet\ar@/^1.3pc/[ddr]&&&&&&\\&&&&\bullet_{v}\ar@/_0.8pc/[ddrr]&&&&...\ar@/_0.8pc/@{-->}[uull]&\ar@{->}^{\text{in-splitting at $v$ }}[rrrr]&&&&&&&&&\bullet_{v_1}\ar@/_0.8pc/[ddrr]&&&&...\ar@/_0.8pc/@{-->}[uull]\\&&&&&&&&&&&&&&&&&\bullet^{v_2}\ar@/_1.5pc/[drrr]&&&&&\\&&&&&&\bullet\ar@/_0.8pc/@{-->}[uurr]&&&&&&&&&&&&&&\bullet\ar@/_0.8pc/@{-->}[uurr]&&}$$

By doing this repeatedly, we have a graph whose trails all have length $1$, thus finishing the proof.
\end{proof}

As a corollary of Proposition \ref{norform-prop}, we obtain the following useful result.

\begin{cor}\label{norform-cor}
 For every graph $E$ of Gelfand-Kirillov dimension three, there is a graph $F$ of Gelfand-Kirillov dimension three in normal form such that  $T_E\cong T_F$ as $\mathbb{Z}$-monoids.
\end{cor}
\begin{proof}
By Proposition \ref{norform-prop}, $E$ can be transformed into a graph $F$ of Gelfand-Kirillov dimension three in normal form by a sequence of in-splittings, out-splittings,  in-amalgamations, and out-amalgamations. Since graphs of Gelfand-Kirillov dimension three are finite and have no sinks, \cite[Theorems 4.4 and 4.7]{Luiz} shows that in-splittings, out-splittings,  in-amalgamations, and out-amalgamations preserve the talented monoids. From these observations, we immediately get that  $T_E\cong T_F$ as $\mathbb{Z}$-monoids, thus finishing our proof.
\end{proof}	

The following proposition plays an important role in proving the main result of this section.

\begin{prop}\label{prop-functor}
Let $E$ be a graph of Gelfand-Kirillov dimension three with all source cycles $C_1^E, \ldots, C_m^E$ and all sink cycles $D_1^E,\ldots, D_n^E$. For any pair $(C^E_i, D^E_j)$, let $\textnormal{Trail}_E(C^E_i, D^E_j)$ be the set of all trails of $E$ connecting $C^E_i$ to $D^E_j$. For any $\alpha \in \textnormal{Trail}_E(C^E_i, D^E_j)$ with the source has index $a_i$ on $C^E_i$ and the range has index $b_j$ on $D^E_j$, let $f_E(\alpha) = b_j - (a_i+ |\alpha|) \pmod{(p^E_i, q^E_j)}$, where $p^E_i = |C^E_i|$ and $q^E_j = |D^E_j|$. Let $G$ be the graph obtained from $E$ using the out-splitting, or in-splitting, or their inverses. Then, the following statements hold:

$(1)$ $G$ is a graph of Gelfand-Kirillov dimension three such that there exists an order-preserving bijection $\phi: \mathcal{C}_E\longrightarrow\mathcal{C}_G$ such that $|c| = |\phi(c)|$ for all $c\in \mathcal{C}_E$;

$(2)$ For any pair $(C^E_i, D^E_j)$, there exists a bijection $$\Theta_{E, G}: \textnormal{Trail}_E(C^E_i, D^E_j)\longrightarrow \textnormal{Trail}_G(\phi(C^E_i), \phi(D^E_j))$$
such that $f_E(\alpha) = f_G(\Theta_{E, G}(\alpha))$ for all $\alpha\in \textnormal{Trail}_E(C^E_i, D^E_j)$.
\end{prop}	
\begin{proof}
(1) By Theorem \ref{willimove}, $A_E\sim_{SSE} A_G$, where $A_E$ and $A_G$ are, respectively,  the adjacency matrices of $E$ and $G$, and so $A_E\sim_{SE} A_G$. By Theorem \ref{corcycle1}, $G$ is a graph of Gelfand-Kirillov dimension three such that there exists an order-preserving bijection $\phi: \mathcal{C}_E\longrightarrow\mathcal{C}_G$ such that $|c| = |\phi(c)|$ for all $c\in \mathcal{C}_E$, where $\mathcal{C}_E$ is the partially ordered set of all disjoint cycles in $E$.

(2) Let $\alpha$ be an arbitrary element of $\textnormal{Trail}_E(C^E_i, D^E_j)$ with the source has index $a_i$ on $C^E_i$ and the range has index $b_j$ on $D^E_j$, $v$ an arbitrary vertex in $E$, and $G$ a graph obtained from $E$ by using one of Moves (I) and (O) and their inverses at $v$. If $v\notin \alpha^0$, then $\alpha$ does not change, and we define $\Theta_{E, G}(\alpha) = \alpha\in \textnormal{Trail}_G(\phi(C^E_i), \phi(D^E_j))$. In this case, it is obvious that  $f_E(\alpha) = f_G(\Theta_{E, G}(\alpha))$. Consider the following cases.

{\it Case} $1$: $v\in \alpha^0\setminus\{s(\alpha), r(\alpha)\}$ and $G$ is obtained from $E$ using the out-splitting  at $v$ with a partition $ \{\mathscr{E}^v_1,\ldots, \mathscr{E}^v_t\} $ of $s^{-1}(v)$. We then have $G^0 = E^0\setminus\{v\} \cup \{v^1, v^2, \ldots, v^t\}$ and $G^1 = E^1\setminus r^{-1}(v) \cup \{e^1, \ldots, e^t\mid e\in r^{-1}(v)\}$. Write $\alpha = \alpha_1eg\alpha_2$, where $\alpha_1, \alpha_2\in E^*$, $e, g\in E^1$, $r(\alpha_1)= s(e)$, $r(e) = s(g) =v$ and $r(g) = s(\alpha_2)$. Let $k$ ($1\le k\le t$) be a unique integer such that $g\in \mathscr{E}^v_k$. We define $$\Theta_{E, G}(\alpha) = \alpha_1e^kg\alpha_2\in \textnormal{Trail}_G(\phi(C^E_i), \phi(D^E_j)).$$ In this case, we note that $\phi(C^E_i) = C^E_i$, $\phi(D^E_j) = D^E_j$ and $|\Theta_{E, G}(\alpha)| = |\alpha_1e^kg\alpha_2|=|\alpha|$, and so
it is obvious that  $f_E(\alpha) = f_G(\Theta_{E, G}(\alpha))$. 

{\it Case} $2$: $v= s(\alpha)$ and $G$ is obtained from $E$ using the out-splitting  at $v$ with a partition $ \{\mathscr{E}^v_1,\ldots, \mathscr{E}^v_t\} $ of $s^{-1}(v)$. Then, there is a unique edge $e\in E^1$ such that $r(e) = v$ (we mention that $e$ lies on cycle $C^E_i$), and so 
$G^0 = E^0\setminus\{v\} \cup \{v^1, v^2, \ldots, v^t\}$ and $G^1 = E^1\setminus \{e\} \cup \{e^1, \ldots, e^t\}$. Write $\alpha = g\beta$, where $g\in E^1$, $\beta\in E^*$, $s(g) = v$ and $r(g) = s(\beta).$ Let $l$ ($1\le l\le t$) be a unique integer such that $g\in \mathscr{E}^v_l$. We define $$\Theta_{E, G}(\alpha) = e^lg\beta\in \textnormal{Trail}_G(\phi(C^E_i), \phi(D^E_j)).$$
Let $f$ be the edge on $C^E_i$ with $s(f) = v$, and let $h$ ($1\le h\le t$) be a unique integer such that $f \in \mathscr{E}^v_h$. We have that $\phi(D^E_j) = D^E_j$, and $\phi(C^E_i)$ is the cycle in $G$ obtained  from $C^E_i$ by replacing vertex $v$ and edge $e$ by $v^h$ and $e^h$, respectively. We have that $|\Theta_{E, G}(\alpha)| = |e^lg\beta| = |\alpha| +1$, $r_G(\Theta_{E, G}(\alpha))=r_G(e^lg\beta) = r_E(\alpha)$ and $s_G(\Theta_{E, G}(\alpha)) = s_G(e^lg\beta)$ is  the vertex with index $a_i-1 \pmod{p^E_i}$, and so 
\begin{align*}
f_G(\Theta_{E, G}(\alpha))&=b_j - (a_i -1 + |\alpha| +1) \pmod{(p^E_i, q^E_j)}\\
&= b_j - (a_i + |\alpha|) \pmod{(p^E_i, q^E_j)}\\
& = f_E(\alpha).
\end{align*}
$$\xymatrixrowsep{0.5pc}\xymatrixcolsep{0.5pc}\xymatrix{&&\bullet\ar@/_0.8pc/@{-->}[ddll]&&&&&&...&&&&&&&&\bullet\ar@/_0.8pc/@{-->}[ddll]&&&&&&...&&&&&&\\&&&&&&\bullet\ar@/_0.7pc/[urr]&&&&&&&&&&&&&&\bullet\ar@/_0.7pc/[urr]&&&&&&&&\\...\ar@/_0.8pc/@{-->}[ddrr]&&&& \bullet^v\ar@/_0.8pc/[uull]^f\ar@/^0.8pc/[urr]^g&&&&&\ar@{<->}[rrrr]&&&&&...\ar@/_0.8pc/@{-->}[ddrr]&&&& \bullet^{v^h}\ar@/_0.8pc/[uull]^f&&&&&\\&&&&&&&&&&&&&&&&&&&\bullet^{v^l}\ar@/^1.4pc/[uur]^g&&&&&&&&&\\&&\bullet^u\ar@/_0.8pc/[uurr]^e&&&&&&&&&&&&&&\bullet^u\ar@/_0.8pc/[uurr]^{e^h}\ar@/_1.5pc/[urrr]^{e^l}&&&&&&&&&&&&}$$
\begin{figure}[h]
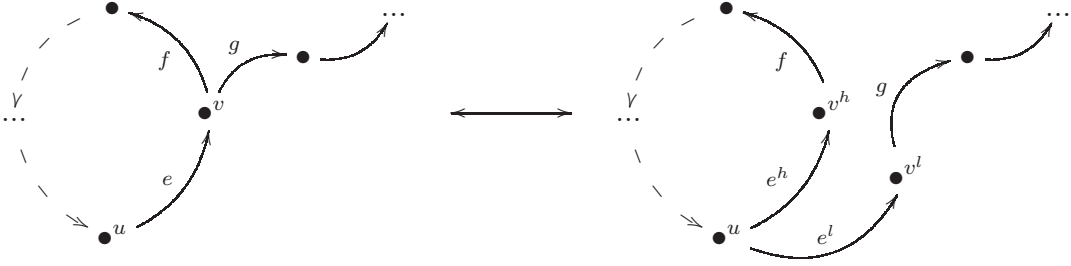
\caption{Out-splitting at $v$ or out-amalgamation.}
\end{figure} 

Inversely, performing out-amalgamations on the trail's source which will move the trail's source to vertex with index  $a_i +1 \pmod{p^E_i}$ and shorten the trail by $1$.	In this case, we have
\begin{align*}
f_G(\Theta_{E, G}(\alpha))&=b_j - (a_i + 1 + |\alpha| -1) \pmod{(p^E_i, q^E_j)}\\
&= b_j - (a_i + |\alpha|) \pmod{(p^E_i, q^E_j)}\\
& = f_E(\alpha).
\end{align*} 

{\it Case} $3$: $v\in \alpha^0\setminus\{s(\alpha), r(\alpha)\}$ and $G$ is obtained from $E$ using the in-splitting  at $v$ with a partition $ \{\mathscr{E}^v_1,\ldots, \mathscr{E}^v_t\} $ of $r^{-1}(v)$.  We then have $G^0 = E^0\setminus\{v\} \cup \{v^1, v^2, \ldots, v^t\}$ and $G^1 = E^1\setminus s^{-1}(v) \cup \{g^1, \ldots, g^t\mid g\in s^{-1}(v)\}$. Write $\alpha = \alpha_1eg\alpha_2$, where $\alpha_1, \alpha_2\in E^*$, $e, g\in E^1$, $r(\alpha_1)= s(e)$, $r(e) = s(g) =v$ and $r(g) = s(\alpha_2)$. Let $k$ ($1\le k\le t$) be a unique integer such that $e\in \mathscr{E}^v_k$. We define $$\Theta_{E, G}(\alpha) = \alpha_1eg^k\alpha_2\in \textnormal{Trail}_G(\phi(C^E_i), \phi(D^E_j)).$$ In this case, we note that $\phi(C^E_i) = C^E_i$, $\phi(D^E_j) = D^E_j$ and $|\Theta_{E, G}(\alpha)| = |\alpha_1eg^k\alpha_2|=|\alpha|$, and so
it is obvious that  $f_E(\alpha) = f_G(\Theta_{E, G}(\alpha))$. 

{\it Case} $4$: $v= r(\alpha)$ and $G$ is obtained from $E$ using the in-splitting  at $v$ with a partition $ \{\mathscr{E}^v_1,\ldots, \mathscr{E}^v_t\} $ of $r^{-1}(v)$. Then, there is a unique edge $g\in E^1$ such that $s(g) = v$ (we mention that $g$ lies on cycle $D^E_j$), and so 
$G^0 = E^0\setminus\{v\} \cup \{v^1, v^2, \ldots, v^t\}$ and $G^1 = E^1\setminus \{g\} \cup \{g^1, \ldots, g^t\}$. Write $\alpha = \beta e$, where $e\in E^1$, $\beta\in E^*$, $r(e) = v$ and $s(e) = r(\beta).$ Let $l$ ($1\le l\le t$) be a unique integer such that $e\in \mathscr{E}^v_l$. We define $$\Theta_{E, G}(\alpha) = \beta e g^l\in \textnormal{Trail}_G(\phi(C^E_i), \phi(D^E_j)).$$
Let $f$ be the edge on $D^E_j$ with $r(f) = v$, and let $h$ ($1\le h\le t$) be a unique integer such that $f \in \mathscr{E}^v_h$. We have that $\phi(C^E_i) = C^E_i$, and $\phi(D^E_j)$ is the cycle in $G$ obtained  from $D^E_j$ by replacing vertex $v$ and edge $g$ by $v^h$ and $g^h$, respectively. We have that $|\Theta_{E, G}(\alpha)| = |\beta eg^l| = |\alpha| +1$, $s_G(\Theta_{E, G}(\alpha))=s_G(\beta eg^l) = s_E(\alpha)$ and $r_G(\Theta_{E, G}(\alpha)) = r_G(\beta e g^l)$ is  the vertex with index $b_j +1 \pmod{q^E_j}$, and so 
\begin{align*}
f_G(\Theta_{E, G}(\alpha))&=b_j + 1- (a_i + |\alpha| +1) \pmod{(p^E_i, q^E_j)}\\
&= b_j - (a_i + |\alpha|) \pmod{(p^E_i, q^E_j)}\\
& = f_E(\alpha).
\end{align*}
$$\xymatrixrowsep{0.5pc}\xymatrixcolsep{0.5pc}\xymatrix{...\ar@/_0.8pc/[drr]&&&&&&\bullet\ar@/_0.8pc/[ddll]_f&&&&&&&&...\ar@/_0.8pc/[drr]&&&&&&\bullet\ar@/_0.8pc/[ddll]_f&&\\&&\bullet\ar@/^0.8pc/[drr]^e&&&&&&&&&&&&&&\bullet\ar@/^1.3pc/[ddr]^e&&&&&&\\&&&&\bullet_{v}\ar@/_0.8pc/[ddrr]_g&&&&...\ar@/_0.8pc/@{-->}[uull]&\ar@{<->}[rrrr]&&&&&&&&&\bullet_{v^h}\ar@/_0.8pc/[ddrr]^{g^h}&&&&...\ar@/_0.8pc/@{-->}[uull]\\&&&&&&&&&&&&&&&&&\bullet^{v^l}\ar@/_1.5pc/[drrr]_{g^l}&&&&&\\&&&&&&\bullet^u\ar@/_0.8pc/@{-->}[uurr]&&&&&&&&&&&&&&\bullet^u\ar@/_0.8pc/@{-->}[uurr]&&}$$
\begin{figure}[h]
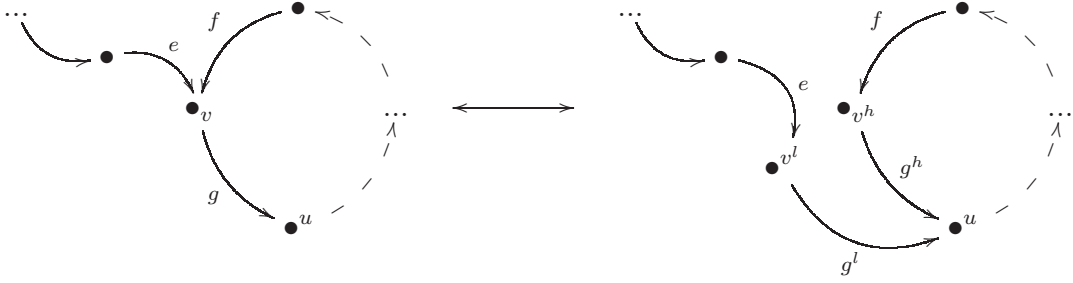
\caption{In-splitting at $v$ or in-amalgamation.}
\end{figure} 

Inversely, performing in-amalgamations on the trail's range which will move the trail's range to vertex with index  $b_i -1 \pmod{q^E_j}$ and shorten the trail by $1$.	In this case, we have
\begin{align*}
f_G(\Theta_{E, G}(\alpha))&=b_j -1 - (a_i + |\alpha| -1) \pmod{(p^E_i, q^E_j)}\\
&= b_j - (a_i + |\alpha|) \pmod{(p^E_i, q^E_j)}\\
& = f_E(\alpha),
\end{align*} 
thus finishing the proof.
\end{proof}	

It is worth mentioning the following fact which is useful to prove the main result of this section.

\begin{prop}\label{prop-lengthen}
Let $E$ be a graph of Gelfand-Kirillov dimension three, $C$ a source cycle of $E$ with $p =|C|$, $D$ a sink cycle of $E$ with $q= |D|$, and let $d:=\operatorname{gcd}(p, q)$. Then, in-splittings, out-splittings, in-amalgamations and out-amalgamations may be used to lengthen or shorten any trail of $E$ from $C$ to $D$ by $d$ edges without changing the source and range of the trail.	Consequently, for every trail $\alpha$ of $E$ from $C$ to $D$ such that $r(\alpha)$ has index $b$ $(1\le b\le q)$, we may shift $r(\alpha)$ to the vertex on $D$ with index $(b - d) \pmod{q}$ without changing the length and source of $\alpha$. 
\end{prop}
\begin{proof}	
Let $\alpha$ be a trail of $E$ from $C$ to $D$. By Bézout's Theorem, there exist positive integers $\widetilde{p}$ and $\widetilde{q}$ such that $\widetilde{p} p-\widetilde{q} q = d$. The move of lengthening or shortening a trail by $d$ edges can thus be attained using in- and out-splittings and amalgamations in the following manner:

(i) Lengthen the trail by $\widetilde{p}p$ edges, by out-splitting at $C$;

(ii) Shorten the trail by $\widetilde{q}q$ edges, by out-amalgamation at $D$.\\ 
Performing the above procedure will not change the source and range of the trail $\alpha$, but repeatedly lengthen or shorten the trail by $d$ edges without changing the source and range of the trail. 

Assume that $r(\alpha)$ is a vertex on $D$ with index $b$ ($1\le b\le q$). We first lengthen the trail $\alpha$ by $d$ edges without changing the source and range of the trail via the above procedure. Then, using in-amalgamation at the cycle $D$ $d$ times repeatedly, we shorten the obtained trail by $d$ edges and shift its range to the vertex on $D$ with index $(b - d) \pmod{q}$, thus finishing the proof.
\end{proof}

Let $E$ be a graph of Gelfand-Kirillov dimension three in normal form with all source cycles  $C_1^E,\ldots , C_m^E$ and all sink cycles $D_1^E,\ldots, D_n^E$. Let $p_i^E := |C^E_i|$ for all $1\le i\le m$ and $q_j^E := |D^E_j|$ for all $1\le j\le n$. We denote
$$ \mathrm{d}_{(i,j)}^E=\operatorname{gcd}\left(p_i^{E}, q_j^{E}\right) \text{ for all } i = 1,2,\ldots, m \text{ and }
j = 1,2,\ldots, n.$$
Let us call the triple $(E, \{C^E_i\}^m_{i=1}, \{D^E_j\}^n_{j=1})$ a {\it pointed graph of Genlfand-Kirillov dimension three}.
For each natural number $c$, we denote by $N_{(i,j)}^E(c)$ be the number of trails of $E$ starting at  $v_i^E$  and ending at a vertex with index $l$ in $D_j^E$ such that $l \equiv c \mbox{ (mod $d_{(i,j)}^E$)}$.

We define a relation on the class of graph of Gelfand-Kirillov dimension three in normal form. For pointed graphs of Gelfand-Kirillov dimension three $(E, \{C^E_i\}^m_{i=1}, \{D^E_j\}^n_{j=1})$ and $(F, \{C^F_i\}^m_{i=1}, \{D^F_j\}^n_{j=1})$, we write $E \approx F$ if there exist $m + n$ integers $a_1, \ldots, a_m,$ $b_1, \ldots, b_n$
such that: \[
\left\{
\begin{array}{ll}
p_i^E = p_i^F,\ q_j^E = q_j^F &\text{ for all } 1 \le i \le m \text{ and } 1 \le j \le n \\ 
N_{(i,j)}^E(c) = N_{(i,j)}^F(c + a_i + b_j)& \text{ for all } c \in \mathbb{N} , 1 \le i \le m ,1 \le j \le  n.
\end{array}  \right
.\]

The following theorem provides us with a number-theoretic criterion for graphs of Gelfand-Kirillov dimension three in normal form to be strong shift equivalent.

\begin{thm}\label{numtheo-crite-sse1}
Let $E$ and $F$ be graphs of Gelfand-Kirillov dimension three in normal form, and let $A_E$ and $A_F$ be the adjacency matrices of $E$ and $F$, respectively.  Then,  $A_E\sim_{SSE} A_F$ if and only if $E \approx F$. 
\end{thm}
\begin{proof}
($\Longrightarrow$)	Assume that $A_E\sim_{SSE} A_F$. By Theorem \ref{willimove}, there exists a graph $E'$ of Gelfand-Kirillov dimension three in normal form transformed from $E$ using a sequence of in-splittings, out-splittings,  in-amalgamations, and out-amalgamations such that $E'$ is isomorphic to  graph $F$. By Proposition \ref{prop-functor}(1), the pointed graphs of $E$ and $E'$ are, respectively, $(E, \{C^E_i\}^m_{i=1}, \{D^E_j\}^n_{j=1})$ and $(E', \{C^{E'}_i\}^m_{i=1}, \{D^{E'}_j\}^n_{j=1})$ with $p_i := |C^E_i| = |C^{E'}_i|$ for all $1\le i\le m$ and $q_j := |D^E_j| = |D^{E'}_j|$ for all $1\le j\le n$. 

Let $(i, j)\in \{1, \ldots, m\}\times \{1, \ldots, n\}$ be an arbitrary pair. By Proposition \ref{prop-functor}(2), there exists a bijection $$\Theta_{E, E'}: \textnormal{Trail}_E(C^E_i, D^E_j)\longrightarrow \textnormal{Trail}_{E'}(C^{E'}_i, D^{E'}_j)$$
such that $f_E(\alpha) = f_{E'}(\Theta_{E, E'}(\alpha))$ for all $\alpha\in \textnormal{Trail}_E(C^E_i, D^E_j)$. Suppose all edges of $E$ starting from $C^E_i$ have the same source $v^E_{a(i)}$ with index $a(i)$ ($1\le a(i)\le p_i$) and all edges of $E'$ starting from $C^{E'}_i$ have the same source $v^{E'}_{a'(i)}$ with index $a'(i)$ ($1\le a'(i)\le p_i$). We claim that 
\begin{equation}
N_{(i,j)}^E(c) = N_{(i,j)}^{E'} (c+a(i) - a'(i) )\ \text{ for all }  c \in \mathbb{Z}. 
\end{equation}
Indeed,  let $d_{(i,j)} = \gcd(p_i,q_j)$ and $c \in \mathbb{Z}.$ 
For every trail $\alpha\in \textnormal{Trail}_E(C^E_i, D^E_j)$, we have $f_E(\alpha) = f_{E'}(\Theta_{E, E'}(\alpha))$. Also, $r_E(\alpha)$ is a vertex on $D^E_j$ with index $c(E) \equiv f_E(\alpha) + a(i) + 1 \pmod{d_{i,j}}$ and $r_{E'}(\Theta_{E, E'}(\alpha))$ is a vertex on $D^{E'}_j$ with index $c(E') \equiv  f_{E'}(\Theta_{E, E'}(\alpha)) + a'(i) + 1 \pmod{d_{i,j}}$.  These observations show that $c(E) - c(E') \equiv a(i)- a'(i) \pmod{d_{i,j}}$. This implies that
\[
N_{(i,j)}^E (c) =N_{(i,j)}^{E'} (c + a(i) - a'(i)) \text{ for all } c \in \mathbb{Z},\] proving the claim.

Since $E'$ is isomorphic to $F$, we may assume that $v_{x}^{E'}$ is transformed to $v_{x+c(i)}^{F} $ for all $v_{x}^{E'} \in (C_i^{E'})^0 $  and $w_{x}^{E'} $ is transformed to $w_{x+d(j)}^{F}$   for all $w_{x}^{E'} \in (D_{j}^{E'})^0$. Therefore, we have \[
N_{(i,j)}^{E'}  (c) = N_{(i,j)}^F (c+ d(j)) \text{ for all } c \in \mathbb{Z},\] 
and so  \[
N_{(i,j)}^E (c) = N_{(i,j)}^F(c+a(i) - a'(i) + d(j))  \text{ for all }  c\in \mathbb{Z},\] thus showing that $E \approx F$.

($\Longleftarrow$) Suppose $(E, \{C^E_i\}^m_{i=1}, \{D^E_j\}^n_{j=1})$ and $(F, \{C^F_i\}^m_{i=1}, \{D^F_j\}^n_{j=1})$ are pointed graphs of Genlfand-Kirillov dimension three such that $E \approx F$.
Then, since $E \approx F$, there exist $m+n$ intergers $a_1,a_2,\ldots ,a_m , b_1,b_2 , \ldots ,b_n$ such that $N_{i,j}^E(c) = N_{i,j}^F(c+a_i+b_j)$ for all $c\in \mathbb{N}$, $1 \le i \le m$ and $1 \le j \le n.$ Let $p_i := p^E_i = p^F_i$ for all $1\le i \le m$ and $q_j := q^E_j = q^F_j$ for all $1\le j \le n$.
We define a graph $G = (G^0, G^1, r_G, s_G)$ as follows: $G^0 = E^0$, $G^1 = E^1$, $s_G = s_E$, and for each $1\le j\le n$, for each edge $e$ which $r_E(e)$ is a vertex with index $k$ in $D^E_j$, $r_G(e)$ is defined to be the vertex with index $(k + b_j) (\text{mod } q_j)$ in $D^E_j$, and $r_G(e) = r_E(e)$ for all other edges $e$. Then, $G$ is obviously a graph of Gelfand-Kirillov dimension in normal form and $N_{i,j}^E(c) = N_{i,j}^{G} (c+ b_j)$ for all $c\in \mathbb{N}$, $1 \le i \le m$ and $1 \le j \le n$, and so $N_{i,j}^{G} (c) = N_{i,j}^F(c+a_i)$ for all $c\in \mathbb{N}$, $1 \le i \le m$ and $1 \le j \le n.$ 

For each $1\le i\le m$, let $v_{x_i}$ be the vertex with index $x_i$ in $C^G_i$ which is the source of all trails of $G$ starting from $C^G_i$. Let $H$ be a graph of Gelfand-Kirillov dimension three in normal form obtained from $G$ by using Moves (I), (O) and their inverses to transform each vertex $v_{x_i}$ to the vertex with index
$(x_i - a_i) \pmod{p_i}$ in $C^D_i$. Then, by formula (2) cited above, we have $N_{i,j}^{H} (c + a_i ) = N_{i,j}^{G} (c)$ for all $c\in \mathbb{N}$, $1\le i \le m$ and $1\le j \le n$, and so $N_{i,j}^{H} (c)= N_{i,j}^F(c)$ for all $c\in \mathbb{N}$, $1\le i \le m$ and $1\le j \le n$. Therefore, without loss of generality, we can assume that  $a_i = b_j = 0$ for all $1\le i \le m$ and $1\le j \le n$, that means, we have $$N_{i,j}^{E} (c)= N_{i,j}^F(c)$$ for all $c\in \mathbb{N}$, $1\le i \le m$ and $1\le j \le n$.

Now, let $\alpha$ be an edge from $C^E_i$ to $D^E_j$ such that $r(\alpha)$  is the vertex on $D^E_j$ with index $t$ $(1\le t\le q_j)$. By Proposition \ref{prop-lengthen}, we may shift $r(\alpha)$ to the vertex on $D^E_j$ with index $r$, where $r$ is the remainder of $t$ when divide $d_{(i,j)} := d^E_{(i,j)} =d^F_{(i,j)}$, without changing the source and length of $\alpha$. Performing this procedure for all trails of $E$ and $F$. Then, we obtain that for every $1 \le i \le m, 1 \le j \le n$ and $1 \le c \le d_{(i,j)}$, there are $N^E_{(i,j)}(c)$ edges of $E$ starting from $C_{i}^E$ and ending at the vertex on $D^E_j$ with index $c$, and there are $N^F_{(i,j)}(c)$ edges of $F$ starting from $C_{i}^F$ and ending at the vertex on $D^F_j$ with index $c$.

For each $1\le i\le m$, let $a^E(i)$ be the index of the vertex on $C^E_i$ which is the source of all trails of $E$ starting from $C^E_i$, and $a^F(i)$ the index of the vertex on $C^F_i$ which is the source of all trails of $F$ starting from $C^F_i$. We define the map $\lambda: E\longrightarrow F$ by setting: for every $1\le j\le n$ and for every vertex $w^E_{b}$ on $D^E_j$ with index $b$ ($1\le b\le q_j$), $\lambda(w^E_{b})$ is defined as the vertex on $D^F_j$ with index $b$. For every $1\le i\le m$ and for every vertex $v^E_{c}$ on $C^E_i$ with index $c$ ($1\le c\le p_i$), $\lambda(v^E_{c})$ is defined as the vertex on $C^F_i$ with index $(c + x(i)) \pmod{p_i}$, where $x(i) = (a^E(i) - a^F(i)) \pmod{p_i}$. Then, since $N_{i,j}^{E} (c)= N_{i,j}^F(c)$ for all $c\in \mathbb{N}$, $1\le i \le m$ and $1\le j \le n$, we immediately get that $\lambda$ is a graph isomorphism. This implies that 
we can transform $E$ and $F$ into a same graph using using in-splittings, out-splittings,  in-amalgamations and out-amalgamations. By Theorem \ref{willimove}, we have $A_E \sim_{SSE} A_F$, this finishing the proof.	
\end{proof}	

It is worth mentioning the following example.

\begin{example} Let $E$ and $F$ be the following two graphs of Gelfand-Kirillov dimension three in normal form.

$$\xymatrixrowsep{1pc}\xymatrixcolsep{1pc}
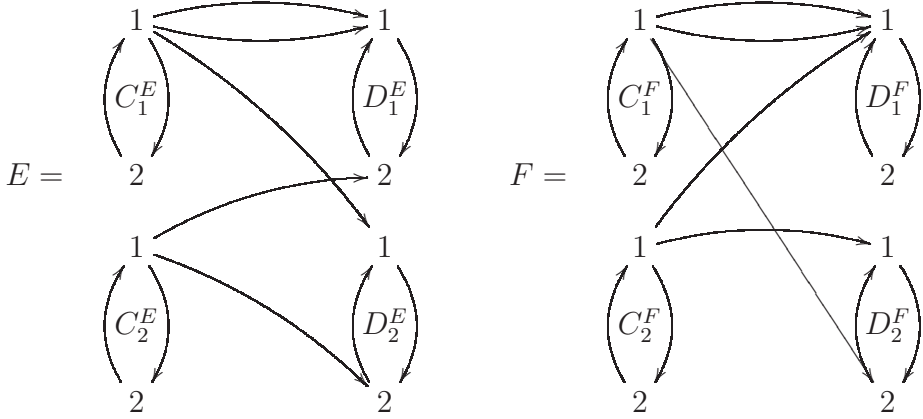
\xymatrix{&1\ar@/^1pc/[dd]\ar@/^0.6pc/[rrrr]\ar@/_0.6pc/[rrrr]\ar@/^0.6pc/[rrrrddd]&&&&1\ar@/^1pc/[dd]&&&1\ar@/^1pc/[dd]\ar@/^0.6pc/[rrrr]\ar@/_0.6pc/[rrrr]\ar[rrrrddddd]&&&&1\ar@/^1pc/[dd]\\&C^E_1&&&&D^E_1&&&C^F_1&&&&D^F_1&\\ E= &2\ar@/^1pc/[uu]&&&&2\ar@/^1pc/[uu]&&F=&2\ar@/^1pc/[uu]&&&&2\ar@/^1pc/[uu]\\&1\ar@/^1pc/[dd]\ar@/^0.6pc/[rrrru]\ar@/^0.6pc/[rrrrdd]&&&&1\ar@/^1pc/[dd]&&&1\ar@/^1pc/[dd]\ar@/^0.6pc/[rrrruuu]\ar@/^0.6pc/[rrrr]&&&&1\ar@/^1pc/[dd]\\&C^E_2&&&&D^E_2&&&C^F_2&&&&D^F_2&\\&2\ar@/^1pc/[uu]&&&&2\ar@/^1pc/[uu]&&&2\ar@/^1pc/[uu]&&&&2\ar@/^1pc/[uu]}$$	
We then have $p^E_i = p^F_i =2$ for all $1\le i\le 2$ and $q^E_j = q^F_j =2$ for all $1\le j\le 2$, and so $d^E_{(i, j)} = 2 = d^E_{(i, j)}$ for all $i, j$. Also, we have $N_{(1,1)}^E(c) = N_{(1,1)}^F(c)$, $N_{(1,2)}^E(c) = N_{(1,2)}^F(c +1)$, $N_{(2,1)}^E(c) = N_{(2,1)}^F(c +1)$ and $N_{(2,2)}^E(c) = N_{(2,2)}^F(c +1)$ for all $c\in \mathbb{N}$.

Assume that $A_E\sim_{SSE} A_F$, that means, $E$ can be  transformed into $F$ by a sequence of in-splittings, out-splittings,  in-amalgamations, and out-amalgamations. By Theorem \ref{numtheo-crite-sse1}, there exist integers $a_1$, $a_2$, $b_1$ and $b_2$ such that $N_{(i,j)}^E(c) = N_{(i,i)}^F(c + a_i + b_j)$ for all $c\in \mathbb{N}$ and for all $1\le i, j\le 2$. This implies that

$$\begin{cases}
a_1 + b_1\equiv 0 \mbox{ (mod 2)}\\
a_1 + b_2\equiv 1 \mbox{ (mod 2)}\\
a_2 + b_1\equiv 1 \mbox{ (mod 2)}\\
a_2 + b_2\equiv 1 \mbox{ (mod 2)},
\end{cases}$$
and so we have $b_2 - b_1\equiv 1 \mbox{ (mod 2)}$ and $b_2 - b_1\equiv 0 \mbox{ (mod 2)}$, a contradiction. Thus, $A_E$ is not strong shift equivalent to $A_F$.
\end{example}

Using Theorems \ref{willimove} and \ref{numtheo-crite-sse1}, we give a number-theoretic criterion for graphs of Gelfand-Kirillov dimension three to be strong shift equivalent.

\begin{thm}\label{numtheo-crite-sse2}
Let $E$ and $F$ be two  graphs of Gelfand-Kirillov dimension three, and let $E'$ and $F'$ be graphs in normal form obtained from $E$ and $F$ using in-splittings, out-splittings,  in-amalgamations and out-amalgamations, respectively. Then $A_E\sim_{SSE} A_F$ if and only if $E' \approx F'$.
\end{thm}
\begin{proof}
By Theorem \ref{willimove}, we always have $A_E\sim_{SSE} A_{F}$ and 	$A_F\sim_{SSE} A_{F'}$, and so  $A_{E}\sim_{SSE} A_{F}$ if and only if $A_{E'}\sim_{SSE} A_{F'}$. By Theorem \ref{numtheo-crite-sse1}, we obtain that $A_{E'}\sim_{SSE} A_{F'}$ if and only if $E' \approx F'$. From these observations, we immediately get that $A_E\sim_{SSE} A_F$ if and only if $E' \approx F'$, thus finishing the proof.
\end{proof}	

\section{Williams' Conjecture for graphs of Gelfand-Kirillov dimension $3$} \label{sec4}
In this section, based on the previous sections, we show that both Williams' Conjecture and The Graded Classification Conjecture hold for the class of graphs of Gelfand-Kirillov dimension three (Theorem \ref{mainthm}). Consequently, we obtain that the Leavitt path algebras of  graphs of Gelfand-Kirillov dimension three are graded Morita equivalent if and only if their graph $C^*$-algebras are equivariant Morita equivalent (Corollary \ref{maintheo-cor1}), as well as show that for two finite graphs $E$ and $F$ of Gelfand-Kirillov dimension three, their adjacency matrices are shift equivalent if and only if the singularity categories $\text{D}_{\text{sg}}(KE/J_E^2)$ and $\text{D}_{\text{sg}}(KF/J_F^2)$ are triangulated equivalent (Corollary \ref{maintheo-cor2}).\medskip

We begin this section by determining all elements of the talented monoid of a graph of Gelfand-Kirillov dimension three in normal form.

\begin{lemma}\label{lem-uni-rep}
Let $E$ be a graph of Gelfand-Kirillov dimension three in normal form  with all source cycles  $C_1^E,\ldots , C_m^E$ and all sink cycles $D_1^E,\ldots, D_n^E$. For each $1\le i\le m$, let $v_i$ be the vertex on $C_i^E$ which is the source of all trails of $E$ starting from $C_i^E$. For each $1\le j\le m$, let $w_j$ be an arbitrary vertex on $D_j^E$. Then for any element $x \in T_{E}$, there exist $c_1,c_2,\ldots ,c_m \in \mathbb{Z}$ such that $x$ may be uniquely written as
\[
\sum \limits_{i=1}^{m} \left( \sum \limits_{0 \le k <p_i}^{} a_{k,i} v_i(c_i - k) \right) + \sum \limits_{j=1}^{n} \left( \sum \limits_{0 \le k <q_j}^{} b_{k,j} w_j(k) \right)
\]  
	where $p_i= \left|C_i^{E}\right|$, $q_j=\left|D_{j}^{E}\right|$ and $a_{k,i}, b_{k,j} \in \mathbb{N}$. 
\end{lemma}
\begin{proof}
Since $T_E$ is a monoid generated by $ \{ v(a), w(b) \mid v \in (C^E_i)^0\ (1 \le i\le m), w\in (D^E_j)^0\ (1\le j \le n), a, b \in \mathbb{Z}\}$,   we can write $x$ in the form:
\begin{equation} x = \sum \limits_{i=1}^{m} \left( \sum \limits_{v \in (C^E_i)^0}^{}  \left( \sum \limits_{a \in \mathbb{Z}}^{} c_{v,a} v(a) \right) \right) +\sum \limits_{j=1}^{n} \left( \sum \limits_{w \in (D^E_j)^0}^{}  \left( \sum \limits_{b \in \mathbb{Z}}^{} d_{w,b} w(b) \right) \right) 
\end{equation}
where $c_{v,a},  d_{w,b}\in \mathbb{N}$, and $c_{v, a} = 0$ and $ d_{w,b} = 0$ for  cofinitely many indices $a$ and $b$, respectively. Let $S_E$ be the submonoid of $T_E$ generated by the set $\{w(k)\mid w\in \bigcup^n_{j=1}(D^E_j)^0 \text{ and } k\in \mathbb{Z}\}$. For each $1\le i \le m$ and for each $v\in (C^E_i)^0$, we have $v(a) = v_i(a + \ell^v_i)$ for all $a\in \mathbb{Z}$, where $\ell^v_i$ is the length of the path in $E$ starting at $v$ and ending at $v_i$. We also have $v_i(a) = v_i(a + p_i) + x(a)$, where $x(a)\in S_E$, for all $a\in \mathbb{Z}$. Let $$c_i:= \max\{a + \ell^v_i + p_i\mid v\in (C^E_i)^0 \text{ and } a\in \mathbb{Z} \text{ with } c_{v,a}\neq 0\}.$$ Let $v$ be a vertex on $(C^E_i)^0$ with $c_{v, a}\neq 0$ and $t$ an integer such that $a + \ell^v_i + tp_i\le c_i$ and $a + \ell^v_i + (t+1)p_i > c_i$. We then have $c_i -p_i < a + \ell^v_i + tp_i\le c_i$ and 
$v(a) = v_i(a + \ell^v_i + tp_i) + x$ for some $x\in S_E$. From these observations and Equation (3), we obtain that

\begin{equation}
x= \sum \limits_{i=1}^{m} \left( \sum \limits_{0 \le k <p_i}^{} a_{k,i} v_i(c_i - k) \right) + y
\end{equation}
for some $y\in S_E$.
	
We note that for each $1\le j\le n$, for each $w\in (D^E_j)^0$ and for every $b\in \mathbb{Z}$, we have $w(b) = w_j(b + \ell^w_j)=  w_j(b + \ell^w_j + tq_j)$ for all $t\in \mathbb{Z}$, where $\ell^w_j$ is the length of the path in $E$ starting at $w$ and ending at $w_j$. Let $t := - \lfloor\frac{b + \ell^w_j}{q_j}\rfloor$. We then have $0\le b + \ell^w_j + tq_j < q_j$ and $w(b) = w_j(b + \ell^w_j + tq_j)$. From this note and Equation (4), we immediately get that 
\[
x=\sum \limits_{i=1}^{m} \left( \sum \limits_{0 \le k <p_i}^{} a_{k,i} v_i(c_i - k) \right) + \sum \limits_{j=1}^{n} \left( \sum \limits_{0 \le k <q_j}^{} b_{k,j} w_j(k) \right),
\]  
as desired. Now we show that this presentation is unique. Assume that we can also write
$$x=\sum \limits_{i=1}^{m} \left( \sum \limits_{0 \le k <p_i}^{} a'_{k,i} v_i(c_i - k) \right) + \sum \limits_{j=1}^{n} \left( \sum \limits_{0 \le k <q_j}^{} b'_{k,j} w_j(k) \right).$$
	
We claim that $a_{k,i}=a_{k,i}^{\prime}$ for all $0 \le i \le m, \ 0 \le k \le p_i-1 $ and $b_{k, j}=b'_{k, j}$ for all $1 \le j \le n, 0 \le k\le q_j-1$. Indeed, consider the $\mathbb{Z}$-order ideal $\langle  w_1, \ldots, w_n\rangle$ of $T_E$ generated by 
$w_1\ldots, w_n$. We note that it is not hard to see that $\langle  w_1, \ldots, w_n\rangle = S_E$. Since $v_i(t) \ge v_i(t+ p_i)$ for all $t\in \mathbb{Z}$ and for all $1\le i\le m$, we have $v_i(t)\notin \langle  w_1, \ldots, w_n\rangle$ for all $t\in \mathbb{Z}$. By \cite[Lemma 3.5]{alfi}, the natural inclusion $T_{\sqcup^m_{i=1} C^E_i} \hookrightarrow T_{E}$ descends to an isomorphism $T_E/\langle  w_1, \ldots, w_n\rangle \cong T_{\sqcup^m_{i=1} C^E_i}$. By Lemma \ref{confuu}, we have $T_{\sqcup^m_{i=1} C^E_i}\cong \bigoplus^m_{i=1}T_{C^E_i}$ as $\mathbb{Z}$-monoids and $T_{C^E_i}\cong \mathbb{N}^{p_i}$ for all $1\le i\le m$. The images of the presentations of $x$ in this quotient monoid are
$$\sum \limits_{i=1}^{m} \left( \sum \limits_{0 \le k <p_i}^{} a_{k,i} v_i(c_i - k) \right) = \sum \limits_{i=1}^{m} \left( \sum \limits_{0 \le k <p_i}^{} a'_{k,i} v_i(c_i - k) \right).$$
Since $T_{\sqcup^m_{i=1} C^E_i}$ is a $\mathbb Z$-cyclic monoid of rank $\sum^m_{i=1}p_i$, it follows that $a_{k,i}=a'_{k,i}$ for all $0 \le i \le m, \ 0 \le k \le p_i-1 $. Since $T_{E}$ is cancellative, we can remove these portions from the representation of $x$ and obtain that
	\begin{equation}
	\sum \limits_{j=1}^{n} \left( \sum \limits_{0 \le k <q_j}^{} b'_{k,j} w_j(k) \right) =\sum \limits_{j=1}^{n} \left( \sum \limits_{0 \le k <q_j}^{} b_{k,j} w_j(k) \right).
	\end{equation}
Since the cycle $D_{j}^{E}$ containing $w_j$ has no exits for all $1\le j\le n$, the set $\bigcup^n_{j=1}(D^E_j)^0 $ form a hereditary saturated subset of $E^0$. Consider the monoid $T_{\sqcup^n_{j=1} D_{j}^{E}}$. Using Lemma \ref{confuu}, and the fact that $\bigcup^n_{j=1}(D^E_j)^0 $ is hereditary and saturated, it is not hard to see that the canonical map $T_{\sqcup^n_{j=1} D_{j}^{E}} \rightarrow T_{E}$ is injective (see, also, \cite[Lemma 3.5]{alfi}). Therefore, the equality (5) also holds in $T_{\sqcup^n_{j=1} D_{j}^{E}}$. However, the monoid $T_{\sqcup^n_{j=1} D_{j}^{E}}\cong \bigoplus^n_{j=1}\mathbb{N}^{q_j}$ is a $\mathbb Z$-cyclic monoid of rank $\sum^n_{j=1}q_j$, and so it follows that $b_{k,j}=b'_{k,j}$ for all $1 \le j \le n$ and $0 \le k\le q_j-1$, thus finishing the proof.
\end{proof}

Following \cite[Subsection 2.2, p. 325]{alfi}, a nonzero element $x$ in a commutative monoid $M$ is called an {\it atom} if $x = y +z$ then $y= 0$ or $z =0$.

The following lemma describes all atoms of the talented monoid of a graph of Gelfand-Kirillov dimension three. 
	
\begin{lemma}\label{lem-atom}
Let $E$ be a graph of Gelfand-Kirillov dimension three in normal form  with all source cycles  $C_1^E,\ldots , C_m^E$ and all sink cycles $D_1^E,\ldots, D_n^E$. For each $1\le j\le m$, let $q_j = |D^E_j|$ and let $w_j$ be an arbitrary vertex on $D_j^E$. Let $x$ be a nonzero element in $T_E$. Then, $x$ is an atom if and only if  there exist two integers $j$ and $k $ such that $1\le j \le n,\ 0 \le k \le q_{j} - 1$ and  $x= w_j(k)$. \end{lemma}
\begin{proof}  
($\Longrightarrow$). Assume that $x$ is an atom in $T_E$. For each $1\le i\le m$, let $v_i$ be the vertex on $C_i^E$ which is the source of all trails of $E$ starting from $C_i^E$. 
 By Lemma \ref{lem-uni-rep}, $x$ can be uniquely written in the form: 
\[
x=   \sum \limits_{i=1}^{m} \left( \sum \limits_{0 \le k <p_i}^{} a_{k,i} v_i(c_i - k) \right) + \sum \limits_{j=1}^{n} \left( \sum \limits_{0 \le k <q_j}^{} b_{k,j} w_j(k) \right),\] 	where $p_i= \left|C_i^{E}\right|$ and $a_{k,i}, b_{k,j} \in \mathbb{N}$. Since $x$ is an atom, we get that $$\sum \limits_{i=1}^{m} \left( \sum \limits_{0 \le k <p_i}^{} a_{k,i} v_i(c_i - k) \right) =0 \text{ or } 
\sum \limits_{j=1}^{n} \left( \sum \limits_{0 \le k <q_j}^{} b_{k,j} w_j(k) \right)=0.$$ If $\sum \limits_{i=1}^{m} \left( \sum \limits_{0 \le k <p_i}^{} a_{k,i} v_i(c_i - k) \right) \neq 0$, then $a_{k, i}\neq 0$ for some $1\le i\le m$ and for some $0\le k\le p_i-1$. We then have $v_i(c_i-k) = v_i(c_i-k + p_i) + y$, where $y$ is a nonzero element of the $\mathbb{Z}$-order ideal $\langle  w_1, \ldots, w_n\rangle$ of $T_E$ generated by 
$w_1\ldots, w_n$, and so $x$ can be represented as a sum of nonzero elements in $T_E$. This implies that $x$ is not an atom, a contradiction. Therefore, we must have $\sum \limits_{i=1}^{m} \left( \sum \limits_{0 \le k <p_i}^{} a_{k,i} v_i(c_i - k) \right) =0$. Then, since $x\neq 0$, we obtain that $x=\sum \limits_{j=1}^{n} \left( \sum \limits_{0 \le k <q_j}^{} b_{k,j} w_j(k) \right)\neq 0$, that means, $\sum \limits_{j=1}^{n} \left( \sum \limits_{k=0}^{q_{j}-1} b_{k,j} \right)$ is a positive integer. If $\sum \limits_{j=1}^{n} \left( \sum \limits_{k=0}^{q_{j}-1} b_{k,j} \right)\ge 2$, then $x$  can be represented as a sum of nonzero elements in $T_E$, and so $x$ is not an atom, a contradiction. Therefore, we have $\sum \limits_{j=1}^{n} \left( \sum \limits_{k=0}^{q_{j}-1} b_{k,j} \right) = 1$, equivalently,  $y= w_j (k)$ for some  $1 \le j \le n , 0 \le k \le q_j - 1.$ 

$(\Longleftarrow)$. It immediately follows from Lemma \ref{lem-uni-rep}, thus finishing the proof.
\end{proof}

We are now in a position to give the main result of this article, showing that both Williams' Conjecture and The Graded Classification Conjecture (Conjecture \ref{conjehfyhtr}) hold for graphs of Gelfand-Kirillov dimension three.

\begin{theorem}\label{mainthm}
Let $E$ and $F$ be essential graphs, where $E$ is a graph of Gelfand-Kirillov dimension three, and let $A_E$ and $A_F$ be the adjacency matrices of $E$ and $F$, respectively.  Let $K$ be an arbitrary field. Then the following are equivalent:
	
$(1)$ The Leavitt path algebras $L_K(E)$ and $L_K(F)$ are graded Morita equivalent;
 
$(2)$ There is an order-preserving $\mathbb{Z}\left[x, x^{-1}\right]$-module isomorphism $K_{0}^{\mathrm{gr}}(L_K(E)) \rightarrow K_{0}^{\mathrm{gr}}(L_K(F))$;

$(3)$ The talented monoids $T_{E}$ and $T_{F}$ are $\mathbb{Z}$-isomorphic;

$(4)$ $A_E\sim_{SE} A_F$;

$(5)$ $A_E\sim_{SSE} A_F$. 	
\end{theorem}
\begin{proof}
$(1) \Longrightarrow (2)$. By \cite[Theorem 2.3.7]{hazbk}, the graded Morita equivalence gives rise to invertible bimodules, which in turn induce an isomorphism on the level of graded $K^{\gr}_{0}$. 

$(2) \Longleftrightarrow (3)$. By \cite{hazli}, the positive cone of the graded Grothendieck group $K_{0}^{\operatorname{gr}}(L_K(E))$ is $\mathcal{V}^{\operatorname{gr}}(L_K(E))$ and $\mathcal{V}^{\operatorname{gr}}(L_K(E)) \cong T_{E}$ as $\mathbb{Z}$-monoids. This implies the equivalence.

$(3) \Longleftrightarrow (4)$. This immediately follows from Corollary \ref{h99} and $(2) \Longleftrightarrow (3)$.

$(4) \Longrightarrow (5)$.  Assume that $A_E\sim_{SE} A_F$. Since $E$ is a graph of Gelfand-Kirillov dimension three and by Theorem \ref{corcycle1}, $F$ is also a graph of Gelfand-Kirillov dimension three, and that the pointed graphs of $E$ and $F$ are, respectively, $(E, \{C^E_i\}^m_{i=1}, \{D^E_j\}^n_{j=1})$ and $(F, \{C^{F}_i\}^m_{i=1}, \{D^{F}_j\}^n_{j=1})$ with $p^E_i = |C^E_i|$, $p^F_i = |C^{F}_i|$ for all $1\le i\le m$ and $q^E_j = |D^E_j|$, $q^F_j = |D^{F}_j|$ for all $1\le j\le n$. By the equivalence of $(3)$ and $(4)$, the talented monoids $T_E$ and $T_E$ are $\mathbb{Z}$-isomorphic to each other. Assume that $\varphi: T_{E} \rightarrow T_{F}$ is a $\mathbb{Z}$-isomorphism. By Corollary \ref{norform-cor}, we may assume that $E$ and $F$ are graphs of Gelfand-Kirillov dimension three in normal form. For each $1\le i\le m$, let $v^E_i$ be the vertex on $C_i^E$ which is the source of all trails of $E$ starting from $C_i^E$, and  for each $1\le j\le m$, let $w^E_j$ be an arbitrary vertices on $D_j^E$. By Lemma \ref{lem-atom}, $w_j^E$ is an atom  in $T_E$ for all $1\le j\le n$, and so  $\phi(w_j^E)$ is an atom  in $T_E$ for all $1\le j\le n$. For each $1\le j\le n$, by Lemma \ref{lem-atom} again,  there is a unique integer $1\le k_j\le n$ such that 
\[\varphi\left(w_{j}^E\right)={ }^{b_j} w_{k_j}^{F}\]
for some  $w^F_{k_j}\in (D^F_{k_j})^0$ and for some $0 \leq b_j \leq q_{k_j}^{F}-1$. Suppose there exist distinct integers $1\le j_1 \neq j_2\le n$ such that $w^F_{k_{j_1}}$ and $w^F_{k_{j_2}}$ lie on the same cycle, say $D^F_{k_{j_1}}$ for example. Then, there exists an integer $c$ such that $w^F_{k_{j_2}}(k) = w^F_{k_{j_1}}(k+c)$ for all $k\in \mathbb{Z}$, and so $$\varphi(w^E_{j_2}(-b_{j_2})) = w^F_{k_{j_2}} = w^F_{k_{j_1}}(c) = \varphi(w^E_{j_1}(c-b_{j_1})),$$
this yields that $w^E_{j_2}(-b_{j_2}) = w^E_{j_1}(c-b_{j_1})$. On the other hand, by the above choice of the vertices $w^E_j$ ($1\le j\le n$) and Lemma \ref{lem-uni-rep}, we have $w^E_{j_2}(-b_{j_2}) \neq w^E_{j_1}(c-b_{j_1})$, a contradiction. Therefore, we get that all the vertices $w_{k_j}^{F}$ ($1\le j\le n$) lie on disjoint cycles. By renumbering all the cycles $D^F_j$ ($1\le j\le n$), without loss of generality, we can assume that 
\begin{equation}
\varphi\left(w_{j}^E\right)={ }^{b_j} w_{j}^{F},
\end{equation}
where $w^F_{j}\in (D^F_{k_j})^0$ and  $0 \leq b_j \leq q_{j}^{F}-1$ for all $1\le j\le n$. 

We note that $q_j^{E}$ and $q^F_j$ are, respectively, the smallest positive integers such that ${ }^{q_j^{E}} w_{j}^{E} = w_{j}^{E}$ and ${ }^{q_j^{F}} w_{j}^{F}=w_{j}^F$. Then, we have $w^F_j(b_j) = \varphi(w^E_j)= \varphi({ }^{q_j^{E}}w^E_j) = w^F_j(b_j + q^E_j)$, and so $q^F_j\le q^E_j$. Furthermore, $\varphi(w^E_j)= w^F_j(b_j) = w^F_j(b_j + q^F_j) =  \varphi({ }^{q_j^{F}}w^E_j)$ implies that $w^E_j = { }^{q_j^{F}}w^E_j$, and so $q^E_j\le q^F_j$, that means, 
we obtain that $q_j^{E}=q_j^{F}=:q_j$.

We next claim that
\begin{equation}
	\varphi\left(v_i^{E}\right)={ }^{a_i} v_i^{F}+\sum \limits_{j=1}^n \left(\sum_{k=0}^{q_j-1} b_{k, j}\left({}^{k} w_{j}^{F}\right) \right) 
\end{equation}
where $b_{k,j} \in \mathbb{N}$ and $v_i^F$ is the vertex on $C_i^F$ which is the source of all trails of $F$ starting from $C_i^F$. Indeed, by Lemma \ref{lem-uni-rep}, there exist $c_1, c_2, \ldots, c_m\in \mathbb{Z}$ such that 
$$\varphi\left(v_i^{E}\right)= \sum \limits_{t=1}^{m} \left( \sum \limits_{0 \le k <p_t}^{} a_{k,t} v_t^F(c_t - k) \right) + \sum \limits_{j=1}^{n} \left( \sum \limits_{0 \le k <q_j}^{} b_{k,j} w_j^F(k) \right),$$ where $a_{k,t}, b_{k,j} \in \mathbb{N}$. Consider the $\mathbb{Z}$-order ideal $\langle  w^E_1, \ldots, w^E_n\rangle$ of $T_E$ generated by $w^E_1\ldots, w^E_n$. We get that for every vertex $v$ on the source cycles of $E$, $v(t)\notin \langle  w^E_1, \ldots, w^E_n\rangle$ for all $t\in \mathbb{Z}$. By \cite[Lemma 3.5]{alfi}, the natural inclusion $T_{\sqcup^m_{i=1} C^E_i} \hookrightarrow T_{E}$ descends to an isomorphism $T_E/\langle  w^E_1, \ldots, w^E_n\rangle \cong T_{\sqcup^m_{i=1} C^E_i}\cong \bigoplus^m_{i=1}\mathbb{N}^{p_i}$ is a $\mathbb Z$-cyclic monoid of rank $\sum^m_{i=1}p_i$. 

Since $\varphi$ induces an isomorphism $\bar{\varphi}:  T_E/\langle  w^E_1, \ldots, w^E_n\rangle\rightarrow T_F/\langle  w^F_1, \ldots, w^F_n\rangle$, passing to the quotient we have that $\bar{\varphi}\left(v_i^{E}\right)=\sum \limits_{t=1}^{m} \left( \sum \limits_{0 \le k <p_t}^{} a_{k,t} v_t^F(c_t - k) \right) $. 
Since $v_i^{E}$ is an atom in $T_{\sqcup^m_{i=1} C^E_i}$, $\bar{\varphi} \left(v_{i}^E \right)$ is also an atom in $T_{\sqcup^m_{i=1} C^F_i}$. By Lemma \ref{lem-atom}, there is a unique integer $1\le k_i\le m$ such that $$\bar{\varphi} \left(v_{i}^E \right) = { }^{a_i} v_{k_i}^{F}$$ where $a_i\in \mathbb{N}$ and $v_{k_i}^F\in (C^E_{k_i})^0$ is the source of all trails of $F$ starting from $C_i^F$. Using the same argument which was done with the case of the elements $\{w^F_{k_j}\}^n_{j=1}$ cited above and replacing $\varphi$ by $\bar{\varphi}$, we get that all the vertices $v_{k_i}^{F}$ ($1\le i\le m$) lie on disjoint cycles. By renumbering all the cycles $C^F_j$ ($1\le 1\le m$), without loss of generality, we can assume that 
$$\bar{\varphi} \left(v_{i}^E \right) = { }^{a_i} v_{i}^{F}$$
where $v^F_{i}\in (C^F_{i})^0$ and  $a_i\in \mathbb{N}$ for all $1\le i\le m$, thus showing Equation (7). Since $p_i^E$ and $p_{i}^F$ are, respectively, the smallest natural numbers such that $^{p_i^E}v_i^E = v_i^E$ and $^{p_i^F} v_{i}^F = v_{i}^F $, it follows that $p_i^E =p_i^F := p_i$. 
	
We now prove that $N_{i,j}^E(c) = N_{i,j}^F(c-a_i+b_j)$ for all $1 \le i \le m, 1 \le j \le n$ and $c\in \mathbb{N}$, where $a_{i}$ $(1\le i\le m)$ and $b_{j}\ (1\le j\le n)$ are the integers defined in Equations (7) and (6), respectively. Indeed,
recall that $T_{E}=M_{\bar{E}}$ is the inductive limit of
$$\cdots \longrightarrow \mathbb{N}^{\sum \limits_{i=1}^m p_i + \sum \limits_{j=1}^n q_j} \xrightarrow{A_{E}} \mathbb{N}^{\sum \limits_{i=1}^m p_i + \sum \limits_{j=1}^n q_j} \xrightarrow{A_{E}} \mathbb{N}^{\sum \limits_{i=1}^m p_i + \sum \limits_{j=1}^n q_j} \longrightarrow \cdots$$
Furthermore, since $E$ is a graph of Gelfand-Kirillov dimension in normal form and by renumbering the vertices of $E$, $A_E$ can be written in the form
	\begin{displaymath}
	A_{E}=\left(\begin{array}{cccc|cccc}C_{p_1} & 0 & \ldots & 0   &  *_{1,1}^E & *_{1,2}^E & \ldots & *_{1,n}^E\\ 0 & C_{p_2} & \ldots & 0  & *_{2,1}^E & *_{2,2}^E & \ldots & *_{2,n}^E \\  \ldots & \ldots & \ldots & \ldots & \ldots & \ldots & \ldots & \ldots   \\
	0 & 0 & \ldots &  C_{p_m} & *_{m,1}^E & *_{m,2}^E & \ldots & *_{m,n}^E  \\ \hline 
	0 & 0 &   \ldots  &  0 & C_{q_1} & 0 & \ldots & 0  \\
	0 & 0 &\ldots & 0    & 0 & C_{q_2}& \ldots & 0  \\
	\ldots & \ldots & \ldots & \ldots & \ldots & \ldots & \ldots & \ldots   \\
	
	0 & 0 &\ldots & 0    & 0 & 0& \ldots & C_{q_n}     \end{array}\right),
	\end{displaymath}
where
$$C_{k}=\left(\begin{array}{cccc}
	0 & 1 & & \cdots \\
	& 0 & 1 & \\
	& & \ddots & \vdots \\
	1 & 0 & \cdots &
	\end{array}\right)$$
denotes the adjacency matrix of the cycle of length $k$, and $*_{i,,j}^E \in M_{p_i\times q_j}(\mathbb{N})$ has at most a nonzero row which is the first one for all $1\le i\le m$ and $1\le j\le n$, and so the matrix $A_{E}$ is invertible (as the source and sink cycles have lengths at least 1). 

For each $k\in \mathbb{Z}$, we denote by $\iota_{E}^{k}: \mathbb{N}^{\sum \limits_{i=1}^m p_i + \sum \limits_{j=1}^n q_j} \rightarrow T_{E}$ the canonical inclusion map associated to the inductive limit. Let $x$ and $y \in \mathbb{N}^{\sum \limits_{i=1}^m p_i + \sum \limits_{j=1}^n q_j} $ such that $\iota_E^k(x) = \iota_E^k (y)$. Let  $\iota_E^k(x) = x' $ and  $\iota _E^k(y) = y'$. By Lemma \ref{confuu}, there exists an element $z\in T_E$ such that $x' \to z$ and $y' \to z$. By choosing large enough interger $d$, we have  $z \to z'$ where  $z'  $ can be written as sum of terms in $ \{v(d)\mid  v \in E^0\} $. Then, $x' \to z'$ and $y' \to z'$ imply that  $A_{E}^{d-k}  x = A_{E}^{d-k}  y$, and so $x = y$.   This implies that  $\iota_{E}^{k}: \mathbb{N}^{\sum \limits_{i=1}^m p_i + \sum \limits_{j=1}^n q_j} \rightarrow T_{E}$  is injective for every $k \in \mathbb{Z}$. 

As the copies of $\mathbb{N}$ in $\mathbb{N}^{\sum \limits_{i=1}^m p_i + \sum \limits_{j=1}^n q_j}$ represent the vertices of $E$, these inclusions are related to the $\mathbb{Z}$-action via the formula 
	\begin{equation}
	\alpha_{1} \circ \iota_{E}^{k}=\iota_{E}^{k+1}, 
	\end{equation}
where $\alpha_{1}: T_{E} \rightarrow T_{E}$ denotes the action of $1 \in \mathbb{Z}$. Similar facts hold for the graph $F$. Let $1\le i\le m$ be an integer, choosing $k$ large enough so that $\im(\iota_{F}^{k}) \supseteq \im(\varphi \circ \iota_{E}^{0})$ and $p_1p_2\ldots p_m q_1 q_2 \ldots q_n \mid k-a_i $ , we obtain an injective map
	$$
	C:=\left(\iota_{F}^{k}\right)^{-1} \circ \varphi \circ\left(\iota_{E}^{0}\right): \mathbb{N}^{\sum \limits_{i=1}^m p_i + \sum \limits_{j=1}^n q_j} \longrightarrow \mathbb{N}^{\sum \limits_{i=1}^m p_i + \sum \limits_{j=1}^n q_j}.
	$$
	(In the equation above, we implicitly restrict the codomains of $\varphi$ and $\iota_{F}^{k}$ to the image of $\iota_{F}^{k}$.) 
	We claim that $C A_{E}=A_{F} C$. On one hand,
	$$
	\begin{aligned}
	\iota_{F}^{k+1} \circ C \circ A_{E} & =\iota_{F}^{k+1} \circ\left(\iota_{F}^{k}\right)^{-1} \circ \varphi \circ\left(\iota_{E}^{0}\right) \circ A_{E}  &\\
	& =\alpha_{1} \iota_{F}^{k} \circ\left(\iota_{F}^{k}\right)^{-1} \circ \varphi \circ\left(\iota_{E}^{0}\right) \circ A_{E}  &\text{ (Equation (8))}  \\
	& =\alpha_{1} \circ \varphi \circ\left(\iota_{E}^{0}\right) \circ A_{E} & \\
	& =\varphi \circ \alpha_{1} \circ\left(\iota_{E}^{0}\right) \circ A_{E} & \text{ ($\varphi$ a $\mathbb{Z}$-isomorphism) }  \\
	& =\varphi \circ\left(\iota_{E}^{1}\right) \circ A_{E} & \text{(Equation (8))}  \\
	& =\varphi \circ\left(\iota_{E}^{0}\right) & \text{($T_{E}$ is a direct limit)} .
	\end{aligned}
	$$ 
	On the other hand,
	$$
	\begin{aligned}
	\iota_{F}^{k+1} \circ A_{F} \circ C & =\iota_{F}^{k+1} \circ A_{F} \circ\left(\iota_{F}^{k}\right)^{-1} \circ \varphi \circ \iota_{E}^{0} & \\
	& =\iota_{F}^{k} \circ\left(\iota_{F}^{k}\right)^{-1} \circ \varphi \circ \iota_{E}^{0}  & \text{ $\left(T_{F}\right.$ as a direct limit)} \\
	& =\varphi \circ \iota_{E}^{0} &
	\end{aligned}
	$$ 
	This shows that $\iota_{F}^{k+1} \circ C \circ A_{E}=\varphi \circ \iota_{E}^{0}=\iota_{F}^{k+1} \circ A_{F} \circ C$. Since $\iota_{F}^{k+1}$ is injective, we immediately get that
	\begin{equation}
	C A_{E}=A_{F} C. 
	\end{equation}
	Now, as $C$ is induced from $\varphi$ and $\iota^k$, the above facts about $\varphi\left(v_{i}^{E}\right)$ and $\varphi\left(w_{j}^{E}\right)$ are reflected in $C$. Since $p_1p_2\ldots p_m q_1 q_2 \ldots q_n \mid k-a_i,$ we have   
	$$
	C =  \left(\begin{array}{ccccc|cccc} \tilde{C}_{a_1 - a_i} & 0 & 0 &\ldots  & 0  &  B_{1,1} & B_{1,2} & \ldots & B_{1,n}\\ 0 & \tilde{C}_{a_2  -a_1} &  0 & \ldots & 0  & B_{2,1} & B_{2,2} & \ldots & B_{2,n} \\  
	\ldots & \ldots &  \ldots & \ldots & \ldots & \ldots & \ldots & \ldots & \ldots   \\
	\ldots & \ldots &  \tilde{C}_{a_i - a_i}& \ldots & \ldots & \ldots & \ldots & \ldots & \ldots   \\
	\ldots & \ldots &  \ldots & \ldots & \ldots & \ldots & \ldots & \ldots & \ldots   \\
	0 & 0 &  0& \ldots & \tilde{C}_{a_m - a_1} & B_{m,1} & B_{m,2} & \ldots & B_{m,n}  \\ \hline 
	0 & 0 &  0 &   \ldots  &  0 & \tilde{C}_{b_1 - a_1} & 0 & \ldots & 0  \\
	0 & 0 & 0  &\ldots & 0    & 0 & \tilde{C}_{b_2 - a_1  }& \ldots & 0  \\
	\ldots & \ldots &  \ldots & \ldots & \ldots & \ldots & \ldots & \ldots & \ldots   \\
	
	0 & 0 & 0 &\ldots & 0    & 0 & 0& \ldots & \tilde{C}_{b_n - a_1}    \end{array}\right)
	$$
	where $\tilde{C}_{a_c - a_i}=\left(C_{p_i}\right)^{a_c - a_i}=\sum \limits_{t=1}^{p_i} E_{t, t + a_c - a_i \bmod p_c}$ and $\tilde{C}_{b_j - a_i}=\left(C_{q_j}\right)^{b_j - a_i}=\sum_{t=1}^{q_j} E_{t, t + b_j - a_i \bmod q_j}$ respectively represent the fact that $\varphi\left(v_i^{E}\right)={ }^{a_i} v_i^{F}+\sum \limits_{j=1}^n \left(\sum_{k=0}^{q_j-1} b_{k, j}\left({}^{k} w_{j}^{F}\right) \right)$ and $\varphi \left( w_j^{E}\right)={ }^{b_j} w_j^{F}$, and $B_{r, s}\in M_{p_r\times q_s}(\mathbb{N})$ for all $1\le r\le m$ and $1\le j\le n$. 
	
	For each $1 \leq t \leq q_j$, write $k_{t}^{E}$ for the number of edges starting from $C_i^E$ to  $D_j^E$ and range (the vertex on the sink-cycle equivalent to) $ ^t w_j^{E}$, and similarly for $F$. Then, the upper right block $*_{i,j}^E$  in $A_{E}$ has at most one nonzero row (the first row): $\left(k_{q_j}^E\ k_{1}^{E}\ k_{2}^{E}\ \cdots \ k_{q_j-1}^{E}\right)$.
	
	Notice that Equation (9) implies that	
	$C_{p_i} B_{i,j}+*_{i,j}^{F} \tilde{C}_{b_j - a_i}=*_{i,j}^{E}+B_{i,j} C_{q_j} $ for any $1 \le j \le n$.
	Let $r:= b_j - a_i$. We  then have
	$$
	\begin{array}{lr}
	B_{2, t}+k_{t-r-1}^{F}=k_{t-1}^{E}+B_{1, t-1} & \text { for } 1 \leq t \leq q_j \\
	B_{x+1, t}=B_{x, t-1} & \text { for } 2 \leq x \leq p_i-1,1 \leq t \leq q_j \\
	B_{1, t}=B_{p_i, t-1} & \text { for } 1 \leq t \leq q_j
	\end{array}
	$$
	We conclude that for any $1 \leq t \leq q_j$, we get that
	$$
	\begin{aligned}
	B_{p_i, t-1}=B_{1, t}=B_{2, t+1}+k_{t-r}^{F}-k_{t}^{E} & =B_{3, t+2}+k_{t-r}^{F}-k_{t}^{E}=\cdots \\
	\cdots & =B_{p_i, t+p_i-1}+k_{t-r}^{F}-k_{t}^{E}
	\end{aligned}
	$$
	where the addition on the indices is performed modulo $q_j$.
	Applying the same argument to $B_{p_i, t+p_i-1}$, we obtain that	
	\begin{align*}
	B_{p_i, t-1} & =B_{p_i, t+2 p_i-1}+k_{t-r}^{F}-k_{t}^{E}+k_{t+p_i-r}^{F}-k_{t+p_i}^{E} \\
	& = B_{p_i, t+3p_i-1} + k_{t-r}^{F}-k_{t}^{E}+k_{t+p_i-r}^{F}-k_{t+p_i}^{E} +  k_{t+2p_i-r}^F - k_{t+2p_i}^E \\
	& \ldots \\
	&= B_{p_i, t+lp_i -1}  + \sum \limits_{x=0}^{l-1} k_{t+xp_i - r}^F - \sum \limits_{x=0}^{l-1} k_{t+xp_i }^E  \tag{*}
	\end{align*}
for all $l \ge 1$.	Write $d_{(i,j)}=\operatorname{gcd}(p_i, q_j)$, and write $p_i=c d_{(i, j)}$ and $q_j=d d_{(i, j)}$, where $\operatorname{gcd}(c, d)=1$. Then
	$$
	B_{p_i, t-1}=B_{p_i, t+c q_j-1}=B_{p_i, t+d p_i -1}.
	$$
Now, replacing $l$ by $d$ in Equation (*), we immediately get that
	\begin{equation}
	\sum_{x=0}^{d-1} k_{x p_i+t-r}^{F}=\sum_{x=0}^{d-1} k_{t+x p_i}^{E}
	\end{equation}
for all $1\le t\le q_j$. Observe that for any $x$ we have $x p_i=x c d_{(i, j)} \equiv 0 \bmod d_{(i, j)}$. For each $1\le t\le q_j$, since the addition on the indices is performed modulo $q_j$, we conclude that the right-hand sum of Equation (10) is precisely the number of trails of $E$ starting from $C_i^E$  and ending at a vertex on $D_j^E$ with index $t \pmod{d_{(i, j)}}$, and the left-hand sum of Equation (10) is precisely the number of trails of $F$ starting from $C_i^F$  and ending at a vertex on $D_j^F$ with index $(t-r) \pmod{d_{(i, j)}}$. Therefore, we  have $$N_{i,j}^E(c) = N_{i,j}^F(c -r) = N_{i,j}^F(c-a_i+b_j)$$  for all $1\le i\le m$, $1 \le j \le n$ and $c\in \mathbb{N}$. By Theorem \ref{numtheo-crite-sse1}, we obtain that $A_E\sim_{SSE} A_F$.

(5) $\Longrightarrow$ (1). Assume that $A_E\sim_{SSE} A_F$. By Theorem \ref{willimove}, $E$ can be obtained from $F$ by a sequence of in-splittings, out-splittings,  in-amalgamations, and out-amalgamations. By \cite[Proposition 15]{hazd}, it follows that $L_K(E)$ is graded Morita equivalent to $L_K(F)$, thus finishing the proof.
\end{proof}

Putting together our result with several established results in the literature, we can now relate the notions of Morita theory in algebras and operator algebras for the class of graphs of Gelfand-Kirillov dimension three. The following result is related to the questions on the relationship between the Morita theory of Leavitt path algebras and graph $C^*$-algebras  (see \cite[Section 5.6]{lpabook} and \cite{eilers}), which shows the equivalence of (2) and (3) of \cite[Conjecture 8.8.2]{CortHaz} for  graphs of Gelfand-Kirillov dimension three.

\begin{cor}\label{maintheo-cor1}
Let $E$ and $F$ be essential graphs, where $E$ is a graph of Gelfand-Kirillov dimension three. Let $K$ be an arbitrary field. Then, $L_K(E)$ is graded Morita equivalent to $L_K(F)$ if and only if the graph $C^*$-algebras $C^*(E)$ and $C^*(F)$ are equivariant Morita equivalent.
\end{cor}
\begin{proof}
($\Longrightarrow$). Assume that $L_K(E)$ is graded Morita equivalent to $L_K(F)$. Since $E$ is a graph of Gelfand-Kirillov dimension three and by Theorem \ref{mainthm}, we have $A_E\sim_{SSE} A_F$, where $A_E$ and $A_F$ are the adjacency matrices of $E$ and $F$, respectively. By Theorem \ref{willimove}, $E$ can be obtained from $F$ by a sequence of in-splittings, out-splittings,  in-amalgamations, and out-amalgamations. Then, by \cite[Theorem 3.2 and Corollary 5.4]{pask},  $C^*(E)$ is strongly Morita equivalent to $C^*(F)$. Moreover, \cite[Sections 2.1 and 2.2]{ef} guarantees that this equivalence is equivariant.

($\Longleftarrow$). Assume that the graph $C^*$-algebras $C^*(E)$ and $C^*(F)$ are equivariant Morita equivalent. We then have $K_0^{\mathbb{T}}(C^*(E))\cong K_0^{\mathbb{T}}(C^*(F))$  as order-preserving $\mathbb{Z}[x, x^{-1}]$-modules (see \cite[p. 297]{comb} and \cite[Proposition 2.9.1]{chrisp}). For any finite graph E, there are
canonical order isomorphisms of For a finite graph $G$, there are
canonical order isomorphisms of $\mathbb{Z}[x, x^{-1}]$-modules
\[K_0^{\gr}(L_K(G)) \cong K_0(\overline{G})\cong K_0(C^*(\overline{G}))\cong K_0^{\mathbb{T}}(C^*(G)),\] where $\overline{G}$ is the covering graph of $G$ (see, e.g., \cite[p. 275]{mathann} and \cite[Proof of Theorem A]{eilers2}). Form these notes, we obtain that $K_0^{\gr}(L_K(E)) \cong K_0^{\gr}(L_K(F))$ as order-preserving $\mathbb{Z}[x, x^{-1}]$-modules. By Theorem \ref{mainthm}, we immediately get that $L_K(E)$ is graded Morita equivalent to $L_K(F)$, thus finishing the proof.
\end{proof}

Let $K$ be a field and $A$ a finite dimensional $K$-algebra. We denote by $A$-mod the category of finitely generated left $A$-modules, and denote by $\text{D}^b(A\text{-mod})$ the bounded derived category of $A$-mod. Recall a complex in $\text{D}^b(A\text{-mod})$ is {\it perfect} provided that it is isomorphic to a bounded complex consisting of projective modules.  The full subcategory consisting of perfect complexes is denoted by $\text{perf}(A)$, which is a triangulated subcategory of $\text{D}^b(A$-$\text{mod})$ and is closed under direct summands (see \cite[Lemma 1.2.1]{Buchweitz}). Following \cite{Orlov}, the {\it singularity category} of $A$ is defined to be the quotient triangulated category \[\text{D}_{\text{sg}}(A)=\text{D}^b(A\text{-mod})/\text{perf}(A).\]
The category $\text{D}_{\text{sg}}(A)$ is a triangulated category.
One can show that the global dimension of $A$  is finite if and only if $\text{D}_{\text{sg}}(A)$ is trivial. In \cite{Serre}, Serre showed that an affine variety $V\subseteq \mathbb{C}^n$ is smooth if and only if the algebra $\mathcal{O}(V)$ of polynomial functions on $V$ satisfies $\text{gl.dim}(\mathcal{O}(V)) < \infty$. In this case, $\dim(V) = \text{gl.dim}(\mathcal{O}(V))$. Consequently, $V$ is singular if and only if $\text{D}_{\text{sg}}(\mathcal{O}(V))$ is non-trivial. Hence, from a homological perspective, $\text{D}_{\text{sg}}(\mathcal{O}(V))$ captures the singularity of $V$.

For a non-commutative algebra $A$, the category $\text{D}_{\text{sg}}(A)$ is a homological invariant of $A$, a
measure on how far $A$ is from having finite global dimension.

Let $E$ be a finite graph. The path algebra $KE$ of $E$ over $K$ is defined as follows. As a $K$-vector space, it has a basis given by all the paths $\text{Path}(E)$ in $E$. For two paths $p$ and
$q$, their multiplication is given by the concatenation $pq$ if $s(p) = t(q)$, and it is zero, otherwise.

It was proved by Chen and Yang~\cite{chen}, using the results of P. Smith~\cite{Smith1} and Hazrat \cite{hazd}, that for the finite graphs $E$ and $F$, $\text{D}_{\text{sg}}(KE/J_E^2)$ is triangulated equivalent to $\text{D}_{\text{sg}}(KF/J_E^2)$ if and only if $L_K(E)$ is graded Morita equivalent to $L_K(F)$, where $J_E$ and $J_F$ are  the two-sided ideals of $KE$ and $KF$ generated by paths in $E$ and $F$ of length greater or equal to $1$, respectively. From this note and \cite[Proposition 15(3)]{hazd}, we immediately get that if $\text{D}_{\text{sg}}(KE/J_E^2)$ is triangulated equivalent to $\text{D}_{\text{sg}}(KF/J_F^2)$, then $A_E$ is shift equivalent to $A_F$. Based on Chen and Yang's work \cite{chen} and Bratteli's classification theorem on the ultramatricial algebras $L_K(E)_0$, one expects that the converse might to be true (see \cite[Conjecture 8.8.2]{CortHaz}). The following result shows that the conjecture holds for the class of graphs of Gelfand-Kirillov dimension three.

\begin{cor}\label{maintheo-cor2}Let $E$ and $F$ be essential graphs, where $E$ is a graph of Gelfand-Kirillov dimension three. Let $K$ be an arbitrary field, and let $KE$ and $KF$ be are the path algebras of $E$ and $F$, respectively. Then, $A_E\sim_{SE} A_F$ if and only if $\textnormal{D}_{\text{sg}}(KE/J_E^2)$ is triangulated equivalent to $\textnormal{D}_{\text{sg}}(KF/J_F^2)$.	
\end{cor}
\begin{proof} 
It immediately follows from Chen and Yang's result cited above and Theorem \ref{mainthm}.
\end{proof}	

\section{Acknowledgements}
The first author and the third author were supported by the Vietnam Academy of Science and Technology.  The second author acknowledges Australian Research Council Discovery Project DP23010318. 





\end{document}